\documentclass[10pt]{article}
\usepackage{comment}
\usepackage{amsmath}
\usepackage{amsthm}
\usepackage{amssymb}
\usepackage{multicol}
\usepackage[active]{srcltx}
\usepackage{algorithm}
\usepackage{array}
\usepackage[dvipsnames,svgnames]{xcolor}
\usepackage{url}
\usepackage{enumitem}
\usepackage{tikz}
\usepackage{graphicx}
\usepackage{subcaption} 
\usepackage{pgfplots}
\pgfplotsset{compat=1.17} % <-- Add this line
\usepackage{tabularx}
\usepackage{booktabs} % for better tables
\usepackage{xcolor}
\DeclareMathOperator{\Prox}{Prox}
\usepackage{comment}
 \usepackage[final, colorlinks = true,
            linkcolor = blue,
            urlcolor  = blue,
            citecolor = blue,
            anchorcolor = blue]{hyperref}
 \usepackage[a4paper,margin=1in,footskip=0.2in]{geometry}          
\providecommand{\U}[1]{\protect \rule{.1in}{.1in}}
\newtheorem{theorem}{Theorem}

\newtheorem{corollary}[theorem]{Corollary}

\newtheorem{definition}[theorem]{Definition}
\newtheorem{example}[theorem]{Example}

\newtheorem{lemma}[theorem]{Lemma}

\newtheorem{proposition}[theorem]{Proposition}
\newtheorem{remark}[theorem]{Remark}

\newcommand{\R}{\mathbb{R}}
\newcommand{\dom}{\operatorname{dom}}
\newcommand{\prox}{\operatorname{Prox}}

\newcommand{\argmin}{\operatorname{argmin}}

\usepackage{orcidlink} 

\usepackage{pgfplots}
\pgfplotsset{compat=1.17}
\usepackage{tikz}

\title{Weakly Quasar-Convex Optimization}%: Theory and Algorithm} 
\author{
Abbas Khademi\orcidlink{0000-0002-8276-6821}
  \thanks{Department of Decision Sciences, HEC Montr\'{e}al, Montr\'{e}al, H3T 2A7, Qu\'{e}bec, Canada.\\ E-mail: abbaskhademi92@gmail.com; abbas.khademi@hec.ca}
  \and
Felipe Lara\orcidlink{0000-0002-9965-0921}
  \thanks{Instituto de Alta Investigaci\'on (IAI), Universidad de Tarapac\'a, Arica, Chile.\\ E-mail: felipelaraobreque@gmail.com; flarao@academicos.uta.cl. Web: felipelara.cl }
}
\date{\today}
\begin{document}
\maketitle

%_____________________%
%     abstract        %
%_____________________%
\begin{abstract}
\noindent 
\small 
This paper introduces the class of weakly quasar-convex functions, a notion that simultaneously generalizes weak convexity (and hence convexity) and quasar-convexity. We establish the theoretical foundation for this new class by developing its basic calculus, deriving characterizations, and systematically investigating its closure and structural properties. Building on these tools, we analyze the proximal point algorithm for minimizing a weakly quasar-convex function under local quadratic growth, proving $Q$-linear convergence of the iterates and $R$-linear convergence of the function values to a global minimizer, with
iteration complexity $O(\ln(\varepsilon^{-1}))$; we also give a fully computable, minimizer-free stopping rule via the proximal residual with matching complexity. Finally, we construct a weakly quasar-convex function, satisfying all our standing assumptions, that is neither weakly convex,
quasi-convex, star quasi-convex, nor quasar-convex for any modulus, showing that our convergence theory applies to instances outside every prior framework in the literature.

\medskip

%_____________________%
%     Keywords        %
%_____________________%
\noindent{\small \emph{Keywords}: Nonconvex optimization; Quasar-convexity; Moreau envelope; Proximal point method; Structured nonconvexity.}

\medskip
%________________%
%     MSC        %
%________________%
\noindent {\bf Mathematics Subject Classification: 90C26, 49J52, 49M37.} 
\end{abstract}
%%==========================%%

%%==========================%%
%%       Introduction       %%
%%==========================%%
\section{Introduction}\label{sec:Introduction}

Many of the most important problems in modern optimization give rise to highly nonconvex objectives, including deep neural network training \cite{liang2022revisiting}, low-rank matrix recovery \cite{li2025low,rahimi2026projected}, phase retrieval \cite{liu2026solving},  robust statistics \cite{la2026robot}, quantum process tomography \cite{quiroga2023using}, and dynamical systems \cite{gao2026inertial,hardt2018gradient}. Since arbitrary nonconvexity rules out general-purpose guarantees, attention has turned to structured nonconvexity, which is weaker than convexity yet strong enough to support rigorous analysis of stochastic gradient methods \cite{sadiev2024stochastic}, momentum schemes \cite{chen2024distributed,khanh2026convergence}, and proximal point methods \cite{grad2024relaxed}. Identifying which structural assumptions yield which guarantees remains an active area of research \cite{cartis2022evaluation}. This paper contributes to this effort by introducing a unified framework that bridges the prominent notions of weak convexity and quasar-convexity.

Identifying tractable problem classes between convexity and full nonconvexity is a central question in optimization. One prominent example is star-convexity, in which a global minimizer remains visible from every point of the domain \cite{NP-2006}. A further relaxation, quasar-convexity (originally called weak quasi-convexity), was proposed in the study of linear dynamical systems \cite{hardt2018gradient}. For a quasar-convex function, the negative gradient always points toward a region of smaller objective value, and this observation has led to near-optimal algorithms for star-convex functions and their extensions \cite{hinder2020near}. In recent years, algorithmic developments based on quasar-convexity have attracted considerable attention \cite{ahookhosh2026quasar,gbaguidi2026accelerated,khademi2025adaptive,nandhan2026boosted,nguyen2026projected,valdivia2025convergence}; see the recent survey \cite{khademi2026convexity} for a broad taxonomy of quasar-convexity and related concepts.

Weak convexity occupies a different corner of this landscape. Dating back to early work in nonsmooth optimization \cite{nurminskii1973quasigradient}, it is characterized by the fact that adding a sufficiently large quadratic penalty to the objective renders it convex. This permits a controlled amount of nonconvexity while preserving enough structure for proximal and subgradient methods, together with smoothing techniques for composite and stochastic problems, to retain rigorous convergence guarantees \cite{davis2019proximally,goujon2024learning,khanh2025inexact,lopez2026projected,silveti2026frank,van2024weak}.

Weak convexity and quasar-convexity describe different forms of
structure in nonconvex optimization. Weak convexity controls
nonconvexity through a uniform quadratic curvature bound, whereas
quasar-convexity imposes a geometric condition relative to a global
minimizer. These complementary viewpoints motivate a common framework
that accommodates both types of structure. We introduce
\emph{weakly quasar-convex} functions through a parametrized inequality
anchored at a global minimizer. This class includes weakly convex functions  and quasar-convex functions, while also admitting functions
outside both classes. We investigate its structural properties and
identify conditions under which it supports convergence of
the proximal point algorithm.

Our contributions are both theoretical and algorithmic and can be
summarized as follows. First, we introduce the class of weakly
quasar-convex functions, defined via a parametrized inequality anchored
at a global minimizer that simultaneously generalizes weak convexity and quasar-convexity, and we prove that this generalization is strict via explicit examples that fail to be weakly convex, quasi-convex, or
quasar-convex. Second, we develop the basic calculus of this new class, together with first-order characterizations, and show that, under a local quadratic growth condition, weak quasar-convexity rules out spurious stationary points and implies a restricted secant inequality (also called star quasiconvexity in \cite{PQK-L}). Furthermore, we prove that the Moreau envelope preserves weak quasar-convexity, with an explicit modulus formula, yielding coercivity and characterization results for the associated proximity operator. Finally, we analyze the proximal point algorithm under local quadratic growth, proving $Q$-linear convergence of the iterates and $R$-linear convergence of the function values to a global minimizer; we illustrate this theory on an example lying outside every prior framework in the literature, where no other existing convergence result applies.

The remainder of the paper is organized as follows. Section~\ref{sec:Preliminaries} recalls notation and preliminary background. Section~\ref{sec:wqc} introduces weakly quasar-convex functions and develops their theoretical foundations, detailing their basic calculus, characterizations, preservation under the Moreau envelope, and the properties of the associated proximity operator. Section~\ref{sec:PPM} establishes convergence of the proximal point algorithm under local quadratic growth. Section~\ref{sec:Conclusion} concludes the paper.

%---------

%%==========================%%
%%     Preliminaries        %%
%%==========================%%
\section{Preliminaries and Basic Definitions}\label{sec:Preliminaries}

The inner product in $\mathbb{R}^{n}$ and the Euclidean norm are denoted by $\langle \cdot,\cdot \rangle$ and $\| \cdot \|$, respectively. We denote  $\mathbb{R}_{+} := [0, + \infty[$, $\mathbb{R}_{++} :=   ]0, + \infty[$, $\overline{\mathbb{R}} := \mathbb{R} \cup\{\pm \infty\}$, and $\mathbb{N}_{0} := \mathbb{N} \cup \{0\}$. 

Given any $x,y,z\in\mathbb{R}^{n}$ and any $\beta \in \mathbb{R}$, the following relations hold
\begin{equation}\label{id:3points}
\langle x-z,y-x \rangle=\frac{1}{2}\|z-y\|^{2}-\frac{1}{2}\|x-z\|^{2}-\frac{1}{2}\|y-x\|^{2},
\end{equation}
\begin{equation}\label{iden:1}
 \|\beta x+(1-\beta)y\|^{2} = \beta\|x\|^{2} + (1-\beta)\|y\|^{2} - \beta(1-\beta) \|x-y\|^{2}.
\end{equation}

Given any extended-valued function $h: \mathbb{R}^{n} \rightarrow \overline{\mathbb{R}}$, the effective domain of $h$ is defined by $\dom(h) := \{x \in \mathbb{R}^{n}: ~h(x) < + \infty \}$. It is said that $h$ is proper if $\dom(h)$ is nonempty and $h(x) > - \infty$ for all $x \in \mathbb{R}^{n}$. The notion of properness is important when dealing with minimization problems.
The set ${\rm epi} (h) := \{(x,t) \in \mathbb{R}^{n} \times \mathbb{R}: ~h(x) \leq t\}$ is the epigraph of $h$, %$S_{\delta} (h) := \{x \in \mathbb{R}^{n}: h(x) \leq \delta\}$ is the sublevel set of $h$ at the height $\delta \in \mathbb{R}$, 
and  $\argmin_{ \mathbb{R}^{n}} h$ is the set of all minimal points of $h$. A function $h$ is lower semicontinuous at $\overline{x} \in \mathbb{R}^{n}$ if for any sequence $\{x^k\}_{k} \subseteq \mathbb{R}^{n}$ with $x^k \rightarrow \overline{x}$, $h(\overline{x}) \leq \liminf_{k \rightarrow + \infty} h(x^k)$. %Furthermore, the current convention $\sup_{\emptyset} h := - \infty$ and $\inf_{\emptyset} h := + \infty$ is adopted.

For $h \colon \mathbb{R}^n \to \mathbb{R}$, $\mathcal{K} \subseteq \mathbb{R}^n$, $\overline{x} \in \mathcal{K}$, 
and $\zeta > 0$, the $\zeta$-\textit{level set} of $h$ over $\mathcal{K}$ at $\overline{x}$ is
$$\mathcal{L}_{\zeta}^h(\mathcal{K},\overline{x}) := \bigl\{x \in \mathcal{K} :~ h(x) \leq h(\overline{x}) + \zeta \bigr\}.$$

A function $h$ satisfies a {$\mu$-local Quadratic Growth (QG) condition} with parameter $\zeta>0$ if
\begin{equation}\label{eq:QG}
 h(x)-h(\overline{x}) \geq \tfrac{\mu}{2}\|x-\overline{x}\|^2, ~ \forall  x\in \mathcal{L}_{\zeta}^h (\mathbb{R}^n,\overline{x}). 
\end{equation}
In the literature, there exist several interesting examples of functions that satisfy the (local) quadratic growth condition (e.g., \cite[Example~E.5]{takahashi2026fast}).

Let $h: \mathbb{R}^n \to \mathbb{R}$ be a differentiable function and let $\overline{x} \in \argmin_{\mathbb{R}^n} h$. A point $x \in \mathbb{R}^n$ is called a spurious stationary point of $h$ (with respect to $\overline{x}$) if $\nabla h(x) = 0$ but $x \neq \overline{x}$.

Let us recall the following generalized convexity notions, which have been introduced and developed in virtue of its interesting properties in machine learning and continuous optimization (see \cite{GGK,hinder2020near,NP-2006,WW}).

\medskip

Let $h: \mathbb{R}^{n} \rightarrow \overline{\mathbb{R}}$ be a proper function with $\overline{x} \in \argmin_{\mathbb{R}^{n}}h$. Then $h$ is said to be:
\begin{itemize}
 \item[$(i)$] (Strongly) star-convex \cite{NP-2006} if there exists $\rho \geq 0$ such that for every $y \in \dom(h)$ and every $\lambda \in [0, 1]$, we have
 \begin{equation*}\label{star:convex}
  h(\lambda \overline{x} + (1-\lambda)y) \leq \lambda h(\overline{x}) + (1 - \lambda) h(y) - \lambda (1 - \lambda) \tfrac{\rho}{2} \| y - \overline{x} \|^{2}.
 \end{equation*}

 \item[$(ii)$] (Strongly) quasar-convex \cite{GGK,hinder2020near} if there exists $\gamma \in  ]0, 1]$ and $\rho \geq 0$ such that for every $y \in \dom(h)$, we have
 \begin{equation*}%\label{eq:qconvex}
  h(\lambda \overline{x} + (1-\lambda)y) \leq \lambda \gamma h(\overline{x}) + (1 - \lambda \gamma) h(y) - \lambda \left(1 - \tfrac{\lambda}{2 - \gamma} \right) \tfrac{\gamma \rho}{2} \| y - \overline{x} \|^{2}. \notag
 \end{equation*}
\end{itemize}
Every (strongly) star-convex function is (strongly) quasar-convex, while the reverse statement does not hold in general (see \cite[Example 8]{de2025extending} with $0<p<1$).

Furthermore, it follows from the definition that $\argmin_{\mathbb{R}^{n}}h$ is a singleton when $h$ is strongly ($\rho > 0$) quasar-convex (resp. star-convex). Moreover, if $\gamma = 1$, then a quasar-convex function reduces to a strongly star-convex function ($\rho > 0$) and star-convex function ($\rho = 0$).

In the differentiable case, we have the following useful characterization for (strongly) quasar-convex functions. Let $\mathcal{K} \subseteq \mathbb{R}^{n}$ be a 
convex and open set and $h: \mathbb{R}^{n} \rightarrow \mathbb{R}$ be a differentiable function with $\argmin_{\mathcal{K}}h \neq \emptyset$. Then $h$ 
is (strongly) quasar-convex with modulus $\gamma \in  ]0, 1]$ and $\rho \geq 0$ on $\mathcal{K}$ if and only if (see \cite[Lemma 10]{hinder2020near})
\begin{equation}\label{char:qconvex}
 h(\overline{x}) \geq h(y) + \tfrac{1}{\gamma} \langle \nabla h(y), \overline{x} - y \rangle + \tfrac{\rho}{2} \| y - \overline{x} \|^{2}, ~ \forall  y \in \mathcal{K}.
\end{equation}

Given a nonempty, closed, and convex set $\mathcal{K} \subset \mathbb{R}^{n}$, the proximity operator on $\mathcal{K}$ of parameter $\beta > 0$ of a proper lower semicontinuous function $h: \mathbb{R}^{n} \rightarrow \overline{\mathbb{R}}$ at $z \in \mathbb{R}^{n}$ is defined as the operator ${\rm Prox}_{\beta h} (\mathcal{K}, \cdot): \mathbb{R}^{n} \rightrightarrows \mathbb{R}^{n}$ for which
\begin{equation}\label{prox:oper}
 {\rm Prox}_{\beta h} (\mathcal{K}, z) := \argmin_{x \in \mathcal{K}} \left\{ h(x) + \tfrac{1}{2 \beta} \Vert z - x \|^{2} \right\}.
\end{equation}
When $\mathcal{K} = \mathbb{R}^{n}$, we simply write ${\rm Prox}_{\beta h} (z) := {\rm Prox}_{\beta h} (\mathbb{R}^{n}, z)$. If $h$ is proper, lower semicontinuous, and convex, then ${\rm Prox}_{\beta h}$ turns out to be a single-valued operator (see, for instance, \cite[Proposition 12.15]{BC-2}). Closely related is the \textit{Moreau envelope} $h^{\beta} \colon \mathbb{R}^n \to \mathbb{R}$ of $h$ with parameter $\beta > 0$, defined by
\begin{equation}
 h^{\beta}(x) := \inf_{y \in \mathbb{R}^n}\left\{ h(y) + \tfrac{1}{2\beta}\|x-y\|^2 \right\}.
\end{equation}

Let us recall some notions related to the Clarke directional derivative \cite{Clarke1990}. Let $h: \R^n \to \overline{\R}$ be proper and locally Lipschitz around $x\in\dom(h)$. Then the Clarke directional derivative of $h$ at $x$ in direction $d \in \mathbb{R}^n$ is defined by
\begin{equation*}
 h^\circ (x; d) = \mathop{\limsup} \limits_{ \substack{y \to x \\  t \downarrow 0}} \frac{h(y+td) - h(y)}{t},
\end{equation*}
 and the Clarke subdifferential of $h$ at $x$ is given by
 \begin{equation*}%\label{def:clarke} 
  \partial^{C} h(x) := \{v \in \mathbb R^n: ~  \langle v, d \rangle \leq h^\circ(x; d), ~ \forall  d \in \mathbb R^n\},
 \end{equation*} 
By convention, $\partial^C h(x):=\emptyset$ for $x\notin\dom(h)$. Moreover,
\begin{equation}\label{prop:direc}
 h^{\circ} (x; d) := \max_{\xi \in \partial^{C} h(x)} \langle \xi, d \rangle, ~ \forall  d \in \mathbb R^{n}.
\end{equation} 
 A point $x \in \dom(h)$ is called a Clarke critical point of $h$ if $0 \in \partial^{C} h(x)$.

\begin{definition} \label{def:q-r-linear}
 Let $\{x^k\}_{k} \subset \mathbb{R}^n$ be a sequence converging to $x^* \in \mathbb{R}^n$. It is said that:
\begin{enumerate}
 \item[$(a)$] $\{x^k\}_{k}$ converges $Q$-linearly to $x^*$ if there exists a constant $r \in   ]0, 1[$ such that for all sufficiently large $k$, we have
 \begin{equation} \label{eq:q-linear}
  \|x^{k+1} - x^*\| \leq r \|x^k - x^*\|.
 \end{equation}

 \item[$(b)$] $\{x^k\}_{k}$ converges $R$-linearly to $x^*$ if there exists a constant $r \in   ]0, 1[$ and a constant $C > 0$ such that for all sufficiently large $k$, we have
 \begin{equation}\label{eq:r-linear}
  \|x^k - x^*\| \leq C r^k.
 \end{equation}
\end{enumerate}
Clearly, $Q$-linear convergence implies $R$-linear convergence, but the converse is not 
true in general. 
\qed
\end{definition}

%%==========================%%
%%        Quasar            %%
%%==========================%%
\section{Weak Quasar-Convexity}\label{sec:wqc}

In this section, we introduce and develop the notions of weak quasar-convexity, a notion that extends weak convexity and quasar-convexity, simultaneously.

%________________________%
%     definition         %
%________________________%
\begin{definition}[Weakly Quasar-Convex]\label{def:wquasar-convex}
A function $h: \R^n \to \R$ is said to be $(\rho, \gamma)$-weakly quasar-convex with respect to $\overline{x} \in \argmin_{\R^n} h$ if for every $x \in \R^n$ and every $\lambda \in [0,1]$, we have
\begin{equation}\label{eq:quasar-wconvex}
 h(\lambda \overline{x} + (1-\lambda)x) \leq \gamma \lambda h(\overline{x}) + (1-\gamma \lambda) h(x) + \frac{\gamma\rho}{2} \lambda \left(1-\frac{ \lambda}{2 -\gamma}\right)\|x - \overline{x}\|^2,
 \end{equation}
where $\gamma \in ]0,1]$ and $\rho \geq 0$. 
\qed
\end{definition}

Definition \ref{def:wquasar-convex} can be extended to to any proper function $h: \mathbb{R}^n \to \overline{\mathbb{R}}$, with $\overline{x}\in \argmin_{\mathbb{R}^n}\,h$, once \eqref{eq:quasar-wconvex} is read under the usual conventions of extended arithmetic. In that setting, the inequality is automatically satisfied (and imposes no constraint) whenever $x \notin \dom(h)$, while for $x \in \dom(h)$ it forces $\lambda \overline{x} + (1-\lambda) x \in \dom(h)$ for every $\lambda \in [0, 1]$; that is, weak quasar-convexity forces $\dom(h)$ to be star-shaped with respect to $\overline{x}$.

In particular, when $\gamma = 1$, Definition \ref{def:wquasar-convex} reduces to $\rho$-weak star-convexity, that is,
\begin{equation}\label{eq:swconvex}
 h(\lambda \overline{x} + (1-\lambda)x) \leq \lambda h(\overline{x}) + (1-\lambda)h(x) + \frac{\rho}{2} \lambda(1-\lambda)\|x - \overline{x} \|^2, ~ \forall  x \in \mathbb{R}^{n}, ~ \forall  \lambda \in [0, 1].
\end{equation}
 These definitions generalize standard weak convexity by anchoring the relaxation to a specific minimizer $\overline{x}$ rather than requiring the property to hold for all pairs of points. When $\rho = 0$ and $\gamma = 1$, definitions recover standard star-convexity with respect to $\overline{x}$. When $\rho = 0$, Definition \ref{def:wquasar-convex} connects to quasar-convexity concepts from recent optimization literature \cite{de2025extending, hinder2020near, khademi2026convexity, WW}. Then the following relationship holds
\begin{align}
  \begin{array}{ccccccc}
  {\rm Convex} & \overset{\star}{\Rightarrow}  &  {\rm \text{Quasar-Convex}} \notag \\
  \Downarrow &   & \Downarrow  \notag \\
  {\rm Weakly ~ Convex} & \overset{\star}{\Rightarrow} & {\rm \text{Weakly Quasar-Convex}} \notag 
 \end{array}
\end{align}
%--------------------------%
where ``$\star$" denotes that $\argmin_{ \mathbb{R}^{n}} h \neq \emptyset$ is required.

All the reverse implications do not hold in general. Indeed, examples of quasar-convex and weakly convex functions which are not convex are well-known (see \cite{hinder2020near,NP-2006}). It remains to present functions which is weakly quasar-convex without being neither quasar-convex nor weakly convex for any modulus.

 To that end, we first present a function which is weakly quasar-convex, but it is not quasar-convex for any $\gamma \in  ]0, 1]$.
%________________________%
%       Example          %
%________________________%
\begin{example}\label{example:02}
 Let $\gamma = \tfrac{1}{2}$, $\rho = \tfrac{5}{2}$, and define the function $h : \mathbb{R} \to \mathbb{R}$ by
\begin{equation}\label{eq:example}
 h(x) =  x^6 - 2x^4 + \tfrac{11}{10}x^2 = x^2 \left( \left(x^2-1\right)^2 + \tfrac{1}{10}\right) \geq 0.
\end{equation}
Clearly, $h(x) \geq 0$ for all $x \in \mathbb{R}$, $h(0) = 0$, and $h(x) \to +\infty$ as $|x| \to \infty$, hence $\overline{x} = 0$ is a global minimizer of $h$ on $\mathbb{R}$.

Furthermore, $h$ is $\left(\tfrac{5}{2},\tfrac{1}{2}\right)$-weakly quasar-convex with respect to $\overline{x} = 0$. With $\overline{x}=0$ and $h(\overline{x})=0$, the condition 
\eqref{eq:quasar-wconvex} becomes, for all $x\in\mathbb{R}$ and 
$\lambda\in[0,1]$,
\begin{equation}\label{eq:wqc-to-prove}
    h\!\left((1-\lambda)x\right) 
    \leq \left(1-\tfrac{\lambda}{2}\right)h(x) 
    + \tfrac{5\lambda}{8}\!\left(1-\tfrac{2\lambda}{3}\right)x^2,
\end{equation}
where we used $\gamma=\tfrac{1}{2}$, $\rho=\tfrac{5}{2}$, 
$\tfrac{\gamma\rho}{2}=\tfrac{5}{8}$, and $\tfrac{\lambda}{2-\gamma}=\tfrac{2\lambda}{3}$. Define the gap function
\begin{equation*}
    \Delta(\lambda,x) 
    := \left(1-\tfrac{\lambda}{2}\right)h(x) 
    + \tfrac{5\lambda}{8}\!\left(1-\tfrac{2\lambda}{3}\right)x^2 
    - h\!\left((1-\lambda)x\right).
\end{equation*}
We must show $\Delta(\lambda,x)\geq 0$ for all 
$x\in\mathbb{R}$ and $\lambda\in[0,1]$. Indeed, substituting $h(x)=x^6-2x^4+\tfrac{11}{10}x^2$, we decompose
\begin{equation*}
 \Delta(\lambda,x) = A_6(\lambda) x^6 + A_4(\lambda) x^4 + A_2(\lambda) x^2,
\end{equation*}
where the coefficients for all $\lambda\in[0,1]$ are
\begin{align*}
    A_6(\lambda) &:= \left(1-\tfrac{\lambda}{2}\right) 
                  - (1-\lambda)^6 =-\lambda^6+6 \lambda^5-15 \lambda^4+20 \lambda^3-15 \lambda^2+\tfrac{11}{2}\lambda  \in\left[0,\tfrac{1}{2}+\tfrac{5}{12 \sqrt[5]{12}}\right], \\
    A_4(\lambda) &:= -2\left(1-\tfrac{\lambda}{2}\right) 
                  + 2(1-\lambda)^4 =2 \lambda^4-8 \lambda^3+12 \lambda^2-7 \lambda  \in\left[-\tfrac{11}{8},0\right], \\
    A_2(\lambda) &:= \tfrac{11}{10}\!\left(1-\tfrac{\lambda}{2}\right) + \tfrac{5\lambda}{8}\!\left(1-\tfrac{2\lambda}{3}\right)  - \tfrac{11}{10}(1-\lambda)^2= -\tfrac{91}{60}\lambda^2+\tfrac{91}{40}\lambda  \in\left[0,\tfrac{273}{320} \right].
\end{align*}
Since $x^2 \geq 0$, we can view $\Delta(\lambda,x)$ as a cubic polynomial in $u = x^2 \geq 0$,
\begin{equation*}
 \Delta(\lambda,x) = u \big( A_6(\lambda) u^2 + A_4 (\lambda) u + A_2(\lambda) \big).
\end{equation*}
To show $\Delta(\lambda,x) \geq 0$, it suffices to prove that the quadratic $Q(u) = A_6 (\lambda) u^2 + A_4 (\lambda) u + A_2(\lambda)$ is non-negative for all $u \geq 0$ and all $\lambda \in [0,1]$. Since $A_6 (\lambda) \ge 0$ and $A_4(\lambda) \leq 0$ for all $\lambda \in [0,1]$, consider the quadratic
\begin{equation*}
 Q(u) = A_6 (\lambda)u^2 + A_4 (\lambda)u + A_2(\lambda), ~ \forall  u \ge 0.
\end{equation*}

\noindent\textbf{Case $\lambda=0$.} Here $A_6 (\lambda) = A_4(\lambda) = A_2(\lambda) = 0$, so $Q\equiv 0$.

\noindent\textbf{Case $\lambda\in]0,1]$.} Then $A_6(\lambda) > 0$, so $Q$ is an upward-opening parabola with vertex at
\begin{equation*}
u^*=-\tfrac{A_4(\lambda)}{2A_6(\lambda)}\ge 0,
\end{equation*}
because $A_4\leq 0$. Thus the global minimum of $Q$ over $\mathbb{R}$ already occurs at $u^*\ge 0$, and
\begin{equation}\label{eq:claim1}
Q(u)\ge 0 \ \ \forall u\ge0 
\Longleftrightarrow 
\min_{\mathbb R} Q =Q(u^*)= A_2(\lambda)-\tfrac{A_4(\lambda)^2}{4A_6(\lambda)}\ge 0
\Longleftrightarrow 
A_4(\lambda)^2 \leq 4A_6(\lambda)A_2(\lambda). 
\end{equation}

Hence it suffices to prove $A_4(\lambda)^2 \leq 4A_6(\lambda)A_2(\lambda)$ for all $\lambda\in[0,1]$.
Direct expansion gives, for all $\lambda$, that
$A_6(\lambda)=\lambda B_6(\lambda)$, $A_4(\lambda) = \lambda B_4 (\lambda)$ and $A_2(\lambda) = \lambda B_2 (\lambda)$, where
\begin{align*}
 B_2 (\lambda) = \tfrac{91}{120}(3-2\lambda), ~ B_4 (\lambda)  = -7+12\lambda-8\lambda^2 + 2 \lambda^3, ~\text{and}~ B_6 (\lambda) = \tfrac{11}{2}-15\lambda + 20 \lambda^2 - 15 \lambda^3+6\lambda^4-\lambda^5.
\end{align*}
For $\lambda=0$, \eqref{eq:claim1} is trivial. For $\lambda>0$, \eqref{eq:claim1} is equivalent to
\begin{equation}\label{eq:claim2}
 B_4 (\lambda)^2 \leq 4B_6(\lambda) B_2 (\lambda), ~ \forall  \lambda\in [0,1]. 
\end{equation}
Expanding both sides, 
$$
\begin{aligned}
B_4(\lambda)^2 &= 4 \lambda^6 -32 \lambda^5 + 112 \lambda^4-220 \lambda^3 + 256 \lambda^2 - 168 \lambda+49, 
\\
4B_6 (\lambda) B_2 (\lambda) &= \tfrac{91}{15} \lambda^6 - \tfrac{91}{2} \lambda^5 + \tfrac{728}{5} \lambda^4- \tfrac{1547}{6} \lambda^3 + 273 \lambda^2 -\tfrac{2548}{15} \lambda + \tfrac{1001}{20}.
\end{aligned}
$$
Define
\begin{equation*}
P(\lambda):=60\big(4B_6(\lambda)B_2(\lambda)-B_4(\lambda)^2\big)=124\lambda^6-810\lambda^5+2016\lambda^4-2270\lambda^3+1020\lambda^2-112\lambda+63.
\end{equation*}
Since  \eqref{eq:claim2} is equivalent to $P(\lambda)\ge 0$, it remains to show that
\begin{equation}\label{eq:claim3}
 P(\lambda) \ge 0, ~ \forall  \lambda \in [0,1]. 
\end{equation}
Write $P$ in the Bernstein basis \cite{farouki2012bernstein} of degree $6$ on $[0,1]$,
\begin{equation*}
P(\lambda)=\sum_{k=0}^{6} b_k \binom{6}{k}\lambda^k(1-\lambda)^{6-k}.
\end{equation*}
The coefficients $b_k$ are obtained from the power-basis coefficients $a_i$ of $P$ via the standard conversion
\begin{equation*}
b_k=\sum_{i=0}^{k}\tfrac{\binom{k}{i}}{\binom{6}{i}} a_i, ~a_0=63,~ a_1=-112,~ a_2=1020,~ a_3=-2270,~ a_4=2016,~ a_5=-810,~ \text{and}~ a_6=124.
\end{equation*}
Carrying out this computation gives $b_0=63$,  $b_1 = \tfrac{133}{3}$, $b_2=\tfrac{281}{3}$, $b_3 = \tfrac{195}{2}$, $b_4=\tfrac{1151}{15}$, $b_5 = \tfrac{155}{3}$ and $b_6=31$. Since the Bernstein basis polynomials $\binom{6}{k}\lambda^k(1-\lambda)^{6 - k}$ are non-negative for $\lambda \in [0,1]$, we conclude
\begin{equation*}
 P(\lambda) = \sum_{k=0}^{6} b_k \binom{6}{k} \lambda^k (1-\lambda)^{6-k} \ge 0, ~ \forall  \lambda \in[0,1],
\end{equation*}
which establishes \eqref{eq:claim3}.
%-------------------------------
 
On the other hand, $h$ is not $\gamma$-quasar-convex at $\overline{x} = 0$ for any $\gamma \in  ]0,1]$. Indeed, if $h$ is $\gamma$-quasar-convex at $\overline{x}=0$ (with $h(\overline{x})=0$), then for all $x \in \mathbb{R}$ and all $\lambda \in [0,1]$, we should have $h \left((1-\lambda) x\right) \leq (1-\gamma \lambda) h(x)$. In particular, take $x = 1$ and $\lambda = \tfrac{1}{2}$. Then, $h(\tfrac{1}{2}) \leq \left( 1 - \tfrac{\gamma}{2} \right) h(1)$. Substituting $h(1) = \tfrac{32}{320}$ and $h(\tfrac{1}{2}) = \tfrac{53}{320}$, we get
 $$ \tfrac{53}{320} \leq \left(1-\tfrac{\gamma}{2} \right) \tfrac{32}{320} ~ \implies ~ 53 \leq 32 - 16 \gamma \implies \gamma \leq -\tfrac{21}{16},$$
 a contradiction. Therefore, $h$ is not quasar-convex for any $\gamma \in  ]0, 1]$.
%-----------------------------    
\qed
\end{example}
%------------end example--------------%

The following proposition will be useful.
%________________________%
%       proposition      %
%________________________%

\begin{proposition} \label{prop:wqcx-radial}
 Let $f: \mathbb{R} \to \mathbb{R}$ be a function with $0 \in \argmin_{\mathbb{R}} f$, $f(0)=0$, and let $g: \mathbb{R}^n \to \mathbb{R}$ be a function such that $1 \leq g \leq M$ for some constant $M > 0$. Define,
\begin{equation*}
     h(x) := \begin{cases}
        f(\|x\|) ~ g\left(\tfrac{x}{\|x\|}\right), &  x \in \mathbb{R}^n \setminus \{0\}\\
        0. & x=0
    \end{cases}
\end{equation*}
If $f$ is $(\rho, \gamma)$-weakly quasar-convex with respect to $0$, with $\rho\geq 0$ and $\gamma \in ]0,1]$, then $h$ is $(M \rho, \gamma)$-weakly quasar-convex with respect to $0$.
\end{proposition}
%------------proof--------------%
\begin{proof}
First, we verify that $0 \in \argmin_{\mathbb{R}^n} h$. Since $f(0)=0$ and $0 \in \argmin_{\mathbb{R}} f$, we have $f(t) \ge 0$ for all $t \ge 0$. Also, $g \ge 1 > 0$. Hence, for any $x \in \mathbb{R}^n$, we have
\begin{equation*}
h(x) = f(\|x\|) ~ g\left(\tfrac{x}{\|x\|}\right) \ge 0 = h(0),
\end{equation*}
with equality if and only if $f(\|x\|)=0$. Thus, $0 \in \argmin_{\mathbb{R}^n} h$.

\noindent
Now, fix $y \in \mathbb{R}^n$, $\lambda \in [0,1]$, and let $\overline{x} = 0$. Since $h(0)=0$ and $\overline{x}=0$,
\begin{equation*}
h(\lambda \overline{x} + (1-\lambda)y) = h((1-\lambda)y) = f((1-\lambda)\|y\|) ~ g\left(\tfrac{y}{\|y\|}\right).
\end{equation*}
Since $f$ is $(\rho, \gamma)$-weakly quasar-convex with respect to $0$, we have
\begin{equation*}
f((1-\lambda)\|y\|) \leq (1-\gamma\lambda) f(\|y\|) + \tfrac{\gamma\rho}{2}\lambda\left(1 - \tfrac{\lambda}{2-\gamma}\right)\|y\|^2.
\end{equation*}
Multiplying by $g(y/\|y\|) \ge 0$ and using $g \ge 1$, we obtain
\begin{align*}
 h((1-\lambda)y) & = f((1-\lambda)\|y\|) ~ g \left( \tfrac{y}{\|y\|}\right) \\
 & \leq (1-\gamma\lambda) f(\|y\|) ~ g\left( \tfrac{y}{\|y\|}\right) + \tfrac{\gamma\rho}{2} \lambda \left( 1 - \tfrac{\lambda}{2-\gamma}\right)\|y\|^2 ~ g \left(\tfrac{y}{\|y\|}\right) \\
 & = (1-\gamma\lambda) h(y) + \tfrac{\gamma\rho}{2} \lambda \left(1 - \tfrac{\lambda}{2-\gamma} \right) \|y\|^2 ~ g\left(\tfrac{y}{\|y\|}\right).
\end{align*}
Since $1 \leq g(y/\|y\|) \leq M$, we have $g\left( \tfrac{y}{\|y\|} \right) \leq M$. Thus,
\begin{equation*}
 h((1-\lambda)y) \leq (1- \gamma \lambda) h(y) + \tfrac{ \gamma (M\rho)}{2} \lambda\left(1 - \tfrac{ \lambda}{2 - \gamma} \right)\|y\|^2.
\end{equation*}
This is exactly the $(M\rho, \gamma)$-weakly quasar-convex inequality for $h$ with respect to $0$. Therefore, $h$ is $(M\rho, \gamma)$-weakly quasar-convex with respect to $0$.
\end{proof}
%------------end proof--------------%

Finally, we present the desired example.
%________________________%
%       Example          %
%________________________%
\begin{example}\label{exam:large} {\bf (Strictly Large)}
 Let $f: [-1,1] \to \mathbb{R}$ be the function given by
$$f(s) = 2-\sqrt{|s|},$$
and $g$ be the function of Example \ref{example:02}, i.e., $g(s) = s^2 \left((s^2-1)^2 + \tfrac{1}{ 10}\right) \geq 0$.

Clearly, $f$ is positive on $[-1,1]$, but $f$ is not weakly convex on $[-1,1]$ for any modulus, since the one-sided derivatives at $s=0$ are $\{\pm\infty\}$. Also, since $(s^2-1)^2+\tfrac{1}{10} \geq \tfrac{1}{10} > 0$ for all $s$, we have $g(s)=0$ if and only if $s=0$, so $0$ is the unique global minimizer of $g$ on $\mathbb{R}$. By Example \ref{example:02}, $g$ is $\left( \tfrac{5}{2}, \tfrac{1}{2} \right)$-weakly quasar-convex with respect to $0$, and $g$ is not $\gamma$-quasar-convex with respect to $0$ for any $\gamma \in ~ ]0,1]$.
Set
\begin{equation*}
 \mathcal{A} := \left\{(x_1,x_2)\in\mathbb{R}^2 ~ : ~ x_1 \geq 0,   -x_1 \leq x_2 \leq x_1 \right\},
\end{equation*}
and define the function $h : \mathcal{A}  \to \mathbb{R}$ by
\begin{equation}\label{the:function}
 h(x_1,x_2) :=
 \begin{cases}
    g(x_1) ~ f\left(\dfrac{x_2}{x_1} \right), & (x_1,x_2) \neq (0,0), \\
    0, & (x_1,x_2) = (0,0).
\end{cases}
\end{equation}

Since $g \geq 0$ and $f > 0$ on $[-1,1]$, we have $h \geq 0$ on $\mathcal{A}$. Moreover, for $(x_1,x_2)\in \mathcal{A}$ with $x_1 > 0$, $h(x_1,x_2) = g(x_1)f(x_2/x_1) > 0$ since $g(x_1)>0$ and $f>0$, and for $x_1=0$, the domain constraint of $\mathcal{A}$ forces $x_2=0$. Hence $\overline{x}=(0,0)$ is the unique global minimizer of $h$ on $\mathcal{A}$, i.e., $\{(0, 0)\} = \argmin_\mathcal{A}~h$. Moreover, along the ray $x_2=0$, $h(t,0) = f(0) g(t) = 2g(t) \to 0$ as $t \downarrow 0$, and more generally, since $f$ is bounded on $[-1, 1]$ and $g(x_1) \to g(0)=0$ as $x_1\downarrow 0$,
$$h(x_1,x_2) = g(x_1)~f\left(\frac{x_2}{x_1} \right) \to 0$$
uniformly over $x_2/x_1 \in [-1,1]$ as $x_1 \downarrow 0$, thus $h$ is continuous at the origin.
We claim that $h$ is weakly quasar-convex with respect to $(0,0)$, but is neither weakly convex nor quasi-convex nor quasar-convex on $\mathcal{A}$ for any modulus. Indeed, let us prove this point by point:

\medskip
\noindent
\emph{(i) $h$ is weakly quasar-convex with respect to $(0,0)$:} On any ray through the origin, the ratio $x_2/x_1$ is constant, say $x_2/x_1=c \in [-1,1]$, so restricted to that ray, $h(t,ct) = f(c)~g(t)$ is a positive scalar multiple of $g$, with $f(c) \in [1,2]$. Since $g$ is $\left(\tfrac{5}{2}, \tfrac{1}{2} \right)$-weakly quasar-convex with respect to $0$, multiplying  \eqref{eq:quasar-wconvex} (with $g(0)=0$) by the positive constant $f(c)$ preserves $\gamma = \tfrac{1}{2}$ while scaling $\rho$ by $f(c)$, yielding
\begin{equation*}
h\big((1-\lambda)x_1, (1-\lambda)x_2\big) \leq \left(1-\tfrac{\lambda}{2}\right)h(x_1,x_2) + \frac{\tfrac{1}{2} \cdot f(c) \cdot \tfrac{5}{2}}{2} \lambda\left(1-\tfrac{2\lambda}{3}\right) x_1^2.
\end{equation*}
Since $x = (x_1, c x_1)$, we have $\|x\|^2 = x_1^2(1+c^2) \geq x_1^2$. Using $x_1^2 \leq \|x\|^2$ together with $f(c) \leq 2$, we get $f(c) \cdot \tfrac{5}{2} \leq 5$, which gives
\begin{equation*}
h\big((1-\lambda)x\big) \leq \left(1-\tfrac{\lambda}{2}\right)h(x) + \frac{\tfrac{1}{2} \cdot 5}{2} \lambda\left(1-\tfrac{2\lambda}{3}\right) \|x\|^2.
\end{equation*}
Thus, $h$ is $\left(5, \tfrac{1}{2} \right)$-weakly quasar-convex with respect to $(0, 0)$ on $\mathcal{A}$.

\medskip
\noindent
\emph{(ii) $h$ is not weakly convex on $\mathcal{A}$ for any modulus:} Fix $\rho\geq 0$ and consider, for $s \in ~ ]0, 1]$, the points $(1, -s), (1, s) \in \mathcal{A}$ and their midpoint $\tfrac{1}{2} (1,-s) + \tfrac{1}{2} (1, s) = (1,0)$. Since $f$ is even, $h(1,-s) = h(1,s) = g(1) f(s)$, and $h(1,0) = g(1) f(0)$. Weak convexity with modulus $\rho$ would require
\begin{equation*}
 g (1)f(0) \leq g(1) f(s) + \tfrac{\rho}{2} \cdot \tfrac14(2s)^2 = g(1) f(s)+\tfrac{\rho}{2}s^2,
\end{equation*}
i.e., using $f(0)-f(s)=\sqrt{s}$, we get $g(1)\sqrt{s} \leq \tfrac{\rho}{2}s^2$.

Since $g(1) = 1-2+\tfrac{11}{10} = \tfrac{1}{10} > 0$, dividing by $\sqrt{s}>0$ gives $\tfrac1{10} \leq \tfrac{\rho}{2}s^{3/2}$, which fails as $s\downarrow 0$ for any fixed $\rho\geq 0$. Hence, $h$ is not $\rho$-weakly convex on $\mathcal{A}$ for any $\rho\geq 0$.

\medskip
\noindent
\emph{(iii) $h$ is not quasi-convex, not star quasi-convex, and not quasar-convex on $\mathcal{A}$ for any modulus:} It suffices to consider $h$ restricted to the half-line $(t,0)$, $t\geq 0$, where $h(t,0) = f(0) g (t) = 2g(t)$. Using $g(0)=0$, $g(1) = \tfrac{1}{10} = \tfrac{32}{320}$, and $g \left( \tfrac{1}{2} \right) = \tfrac{53}{320}$ (see Example \ref{example:02}), we get
\begin{equation*}
h(0,0)=0, \quad h(1,0) = \tfrac{1}{5} = \tfrac{32}{160}, \quad\text{and}\quad h\!\left(\tfrac12,0\right) = \tfrac{53}{160}.
\end{equation*}
Since $\tfrac{53}{160} > \tfrac{32}{160}$, the point $\left(\tfrac12,0\right)$, which lies on the segment joining $(0,0)$ and $(1,0)$, satisfies $h\left( \tfrac12, 0 \right) > \max\{h(0,0), h(1,0)\}$, which directly violates quasi-convexity and, in particular, star quasi-convexity with respect to $(0, 0)$ (see \cite[Theorem 12]{PQK-L}).

Furthermore, $h$ is not $\gamma$-quasar-convex with respect to $(0,0)$ for any $\gamma\in~]0,1]$. Indeed, taking $x=(1,0)$ and $\lambda=\tfrac12$ in the quasar-convexity inequality $h((1-\lambda)x)\leq(1-\gamma\lambda)h(x)$ reduces, after cancelling the common factor of $2$ from $h(t,0)=2g(t)$, to $g\!\left(\tfrac12\right)\leq\left(1-\tfrac{\gamma}{2}\right)g(1)$, i.e., $\tfrac{53}{320}\leq\left(1-\tfrac{\gamma}{2}\right)\tfrac{32}{320}$, which is exactly the inequality shown in Example~\ref{example:02} to force $\gamma\leq-\tfrac{21}{16}$, a contradiction.

Therefore, $h$ is weakly quasar-convex on $\mathcal{A}$ with respect to $(0,0)$ (indeed $\left(5,\tfrac12\right)$-weakly quasar-convex), yet neither weakly convex, nor quasi-convex, nor star quasi-convex, nor quasar-convex, for any modulus.
\qed
\end{example}
%------------end example--------------%

%%==========================%%
%%        Properties        %%
%%==========================%%
\subsection{Basic Properties}
We now establish several fundamental properties of weakly quasar-convex functions. 

%________________________%
%       proposition      %
%________________________%
\begin{proposition}\label{prop:quasar-ray}
$h$ is $(\rho,\gamma)$-weakly quasar-convex with respect to $\overline{x} \in \argmin_{\R^n} h$ (where $\rho \geq 0$, $\gamma \in   ]0, 1]$) if and only if for every $d \in \R^n$ with $\|d\|=1$, the function
\begin{equation}\label{eq:quasar-ray}
 \varphi_d(t) := h(\overline{x} + td), \quad t \geq 0,
\end{equation}
is $(\rho,\gamma)$-weakly quasar-convex with respect to $t = 0$ on $[0, +\infty[$, i.e., for every $t \geq 0$ and every $\lambda \in [0,1]$, we have
\begin{equation}
\varphi_d\bigl((1-\lambda)t\bigr) \leq \gamma\lambda  \varphi_d(0) + (1 - \gamma\lambda) \varphi_d(t) + \tfrac{\gamma\rho}{2}  \lambda\left(1 - \tfrac{\lambda}{2-\gamma}\right) t^2. \label{1d:char}
\end{equation}
\end{proposition}
%------------------%
\begin{proof}
($\Rightarrow$) 
Fix $d$ with $\|d\|=1$. For any $t \geq 0$ and $\lambda \in [0,1]$, define $x = \overline{x} + td$. Then,
\begin{align*}
 \varphi_d ((1-\lambda)t) & = h (\overline{x} + (1-\lambda) td) = h(\lambda \overline{x} + (1-\lambda)(\overline{x} + td)) = h(\lambda \overline{x} + (1-\lambda) x).
\end{align*}
Applying the $(\rho,\gamma)$-weak quasar-convexity of $h$ at $\overline{x}$ and since $\|x - \overline{x}\| = \|td\| = t$, we get
\begin{align*}
h\bigl(\lambda \overline{x} + (1-\lambda)x\bigr) 
&\leq \gamma\lambda  h(\overline{x}) + (1 - \gamma\lambda)  h(x) 
   + \tfrac{\gamma\rho}{2}  \lambda\left(1 - \tfrac{\lambda}{2-\gamma}\right) \|x - \overline{x}\|^2 \\
&= \gamma\lambda  \varphi_d(0) + (1 - \gamma\lambda)  \varphi_d(t) 
   + \tfrac{\gamma\rho}{2}  \lambda\left(1 - \tfrac{\lambda}{2-\gamma}\right) t^2,
\end{align*}
i.e., \eqref{1d:char} holds.

\noindent
($\Leftarrow$) 
Let $x \in \R^n$ be arbitrary. If $x = \overline{x}$, the weak quasar-convexity inequality holds trivially (both sides equal $h(\overline{x})$). Otherwise, define
$d = \tfrac{x - \overline{x}}{\|x - \overline{x}\|}$ and $t = \|x - \overline{x}\| > 0$,
hence $\|d\| = 1$ and $x = \overline{x} + td$. Then, for any $\lambda \in [0,1]$,
\begin{align*}
h\bigl(\lambda \overline{x} + (1-\lambda)x\bigr) 
&= h\bigl(\overline{x} + (1-\lambda)td\bigr) = \varphi_d\bigl((1-\lambda)t\bigr).
\end{align*}
Applying relation \eqref{1d:char} to $\varphi_d$ gives
\begin{align*}
\varphi_d\bigl((1-\lambda)t\bigr) 
&\leq \gamma\lambda  \varphi_d(0) + (1 - \gamma\lambda)  \varphi_d(t) 
   + \tfrac{\gamma\rho}{2}  \lambda\left(1 - \tfrac{\lambda}{2-\gamma}\right) t^2 \\
&= \gamma\lambda  h(\overline{x}) + (1 - \gamma\lambda)  h(x) 
   + \tfrac{\gamma\rho}{2}  \lambda\left(1 - \tfrac{\lambda}{2-\gamma}\right) \|x - \overline{x}\|^2,
\end{align*}
i.e., $h$ is $(\rho,\gamma)$-weakly quasar-convex at $\overline{x}$.
\end{proof}
%------------end proof--------------%

%________________________%
%       proposition      %
%________________________%
\begin{proposition}\label{prop:quasar-sum}
Let $h_i: \R^n \to \R$ be $(\rho_i,\gamma)$-weakly quasar-convex with respect to $\overline{x} \in \argmin_{\mathbb{R}^{n}}h_i$ for every $i=1,\cdots,m$, where $\gamma \in   ]0, 1]$ is the same for all $i$. Then for every $\alpha_i \geq 0$, the function
$h := \sum_{i=1}^m \alpha_i h_i$ 
is $\left(\sum_{i=1}^m \alpha_i \rho_i,\gamma\right)$-weakly quasar-convex with respect to $\overline{x}$.
\end{proposition}
%------------------%
\begin{proof}
Since $\overline{x} \in \cap_{i} \argmin_{\mathbb{R}^{n}}h_i$ and $\alpha_i \geq 0$, it follows that $\overline{x} \in \argmin_{\R^n} h$. Then for every $x \in \R^n$ and every $\lambda \in [0,1]$, we have
\begin{align*}
h\bigl(\lambda \overline{x} + (1-\lambda)x\bigr) 
&= \sum_{i=1}^m \alpha_i  h_i\bigl(\lambda \overline{x} + (1-\lambda)x\bigr) \\
&\leq \sum_{i=1}^m \alpha_i \Biggl(  \gamma\lambda  h_i (\overline{x}) + (1 - \gamma\lambda) h_i(x) + \tfrac{\gamma \rho_i}{2}  \lambda \left(1 - \tfrac{ \lambda }{2 - \gamma}\right) \|x - \overline{x}\|^2 \Biggr) \\
&= \gamma\lambda \sum_{i=1}^m \alpha_i h_i(\overline{x}) + (1 - \gamma\lambda) \sum_{i=1}^m \alpha_i h_i(x)  + \tfrac{\gamma}{2}  \lambda\left(1 - \tfrac{\lambda}{2-\gamma}\right) \|x - \overline{x} \|^2 \sum_{i=1}^m \alpha_i \rho_i \\
& = \gamma\lambda  h(\overline{x}) + (1 - \gamma \lambda)  h(x) + \tfrac{\gamma}{2} \left( \sum_{i=1}^m \alpha_i \rho_i\right) \lambda \left(1 - \tfrac{\lambda}{2-\gamma}\right) \|x - \overline{x} \|^2.
\end{align*}
Therefore, $h$ is $(\rho,\gamma)$-weakly quasar-convex at $\overline{x}$ with modulus $\rho = \sum_{i=1}^m \alpha_i \rho_i$ and $\gamma \in ]0, 1]$.
\end{proof}
%------------end proof--------------%

%

The following calculus rule will be useful.

%________________________%
%       proposition      %
%________________________%
\begin{proposition} \label{prop:weak-quasar-shift}
 Let $h: \mathbb{R}^n \to\mathbb{R}$ and let $\overline{x} \in \argmin_{\mathbb{R}^n} h$. For $\gamma \in ]0,1]$ and $\rho \ge 0$, define
 \begin{equation}\label{eq:shifted-function}
  g(x) := h(x) + \tfrac{\gamma \rho}{2 (2-\gamma)} \|x - \overline{x}\|^2.
 \end{equation}
 Then $h$ is $(\rho,\gamma)$-weakly quasar-convex with respect to $\overline{x}$ if and only if $g$ is $\gamma$-quasar-convex with respect to $\overline{x}$.
\end{proposition}
%-----------------
\begin{proof}
 Clearly, $\overline{x} \in \argmin_{ \mathbb{R}^n} g$. Furthermore, throughout the proof, ``$g$ is $\gamma$-quasar-convex with respect to $\overline{x}$'' refers to the $\rho = 0$ instance of \eqref{eq:quasar-wconvex}, namely
\begin{equation}\label{eq:QC}
 g(\lambda \overline{x} + (1-\lambda)x) \leq \gamma \lambda g(\overline{x}) + (1-\gamma\lambda) g(x), ~ \forall  x \in \mathbb{R}^n, ~ \forall  \lambda \in [0,1].
\end{equation}
Set $c := \frac{\gamma\rho}{2(2-\gamma)}$. Since $\gamma \in ]0,1]$, we have $2-\gamma \ge 1 > 0$, so $c$ is well defined and $c \ge 0$ because $\rho \ge 0$. Furthermore, the term $c\|x-\overline{x}\|^{2}$ is non-negative for every $x \in \mathbb{R}^n$, vanishing precisely at $x = \overline{x}$ when $\rho > 0$ (and identically, with $g = h$, when $\rho = 0$). Consequently,
\begin{equation*}
 g(x) - g(\overline{x}) = \big(h(x) - h (\overline{x}) \big) + c\|x - \overline{x}\|^{2} \ge 0, ~ \forall  x \in \mathbb{R}^n,
\end{equation*}
since both summands are non-negative, the first because $\overline{x} \in \argmin_{\mathbb{R}^n} h$. It therefore follows that $\overline{x}$ is likewise a global minimizer of $g$ over $\mathbb{R}^n$, thus the base point $\overline{x}$ remains valid for the shifted function $g$, and the assertion that $g$ is $\gamma$-quasar-convex with respect to $\overline{x}$ is consequently well posed.

Fix $x \in \mathbb{R}^n$ and $\lambda \in [0,1]$, and set $z_\lambda := \lambda \overline{x} + (1-\lambda)x$. Since $z_\lambda - \overline{x} = (1-\lambda)(x - \overline{x})$, we have $\|z_\lambda - \overline{x}\|^{2} = (1-\lambda)^{2}\|x-\overline{x}\|^{2}$, thus
\begin{equation*}
g(\overline{x}) = h(\overline{x}), \quad g(x) = h(x) + c\|x-\overline{x}\|^{2}, \quad\text{and}\quad g(z_\lambda) = h(z_\lambda) + c(1-\lambda)^{2}\|x-\overline{x}\|^{2}.
\end{equation*}
Substituting these three expressions in \eqref{eq:QC} and rearranging shows that $g(z_\lambda) \leq \gamma\lambda g(\overline{x}) + (1-\gamma\lambda) g(x)$ holds if and only if
\begin{equation*}
 h(z_\lambda) \leq \gamma\lambda h(\overline{x}) + (1-\gamma \lambda) h(x) + c \Big((1-\gamma\lambda) - (1-\lambda)^{2}\Big) \|x-\overline{x}\|^{2}.
\end{equation*}
A direct expansion gives $(1-\gamma\lambda) - (1-\lambda)^{2} = \lambda(2-\gamma-\lambda)$, so that
\begin{equation*}
 c \Big((1-\gamma\lambda) - (1-\lambda)^{2}\Big) = \tfrac{\gamma \rho}{2(2-\gamma)} \lambda(2-\gamma-\lambda) = \tfrac{\gamma \rho}{2}\lambda\left(1 - \tfrac{\lambda}{2-\gamma} \right),
\end{equation*}
which is exactly the coefficient of $\|x-\overline{x}\|^2$ appearing in \eqref{eq:quasar-wconvex}. Then, for every $x \in \mathbb{R}^n$ and every $\lambda \in [0,1]$, the inequality $g(z_\lambda) \leq \gamma\lambda g(\overline{x}) + (1-\gamma\lambda)g(x)$ holds if and only if
\begin{equation*}
h(z_\lambda) \leq \gamma\lambda h(\overline{x}) + (1-\gamma\lambda) h(x) + \tfrac{\gamma\rho}{2}\lambda\left(1-\tfrac{\lambda}{2-\gamma}\right)\|x-\overline{x}\|^{2},
\end{equation*}
that is, precisely when inequality \eqref{eq:quasar-wconvex} holds for $h$ at the same pair $(x, \lambda)$. Since both notions of quasar-convexity require their defining inequality to hold for every $x \in \mathbb{R}^n$ and every $\lambda \in [0,1]$, taking the conjunction over all such pairs on each side yields $h$ is $(\rho,\gamma)$-weakly quasar-convex with respect to $\overline{x}$ if and only if $g$ is $\gamma$-quasar-convex with respect to $\overline{x}$, 
which proves the proposition. 
\end{proof}

%------------------%

Now, we provide a sufficient condition for the preservation of the weak quasar-convexity under linear composition.
%________________________%
%       proposition      %
%________________________%
\begin{proposition}\label{prop:quasar-composition}
Let $h: \R^m \to \R$ be $(\rho, \gamma)$-weakly quasar-convex with respect to $\overline{y} \in \argmin h$, with $\rho\geq 0$ and $\gamma \in ]0, 1]$. Let $A: \R^n \to \R^m$ be a linear operator satisfying $A\overline{x} = \overline{y}$, and assume $\overline{x} \in \argmin_{\R^n} g$ where $g(x) = h(Ax)$. Then $g$ is $\bigl( \rho \|A\|_{ \mathrm{op}}^2, \gamma \bigr)$-weakly quasar-convex with respect to $\overline{x}$, where $\|A\|_{ \mathrm{op}}$ denotes the operator norm of $A$.
\end{proposition}
%------------------%
\begin{proof}
 Since $A$ is linear and $A\overline{x} = \overline{y}$, we have for every $x \in \R^n$ and every $\lambda \in [0,1]$ that $A \bigl(\lambda \overline{x} + (1-\lambda)x\bigr) = \lambda A\overline{x} + (1-\lambda)Ax = \lambda \overline{y} + (1-\lambda)Ax$. Applying the $(\rho, \gamma)$-weak quasar-convexity of $h$ at $\overline{y}$ with $y = Ax$ yields
\begin{align}
 g\bigl(\lambda \overline{x} + (1-\lambda)x\bigr) & = h\bigl(A(\lambda \overline{x} + (1-\lambda)x) \bigr) = h\bigl(\lambda \overline{y} + (1-\lambda) Ax\bigr) \notag \\
 & \leq \gamma\lambda  h(\overline{y}) + (1 - \gamma\lambda)  h(Ax) + \tfrac{\gamma\rho}{2}  \lambda\left(1 - \tfrac{\lambda}{2-\gamma}\right) \|Ax - \overline{y}\|^2. \label{eq:quasar-op}
\end{align}
Since $\overline{y} = A\overline{x}$, we have $\|Ax - \overline{y}\| = \|A(x - \overline{x})\|$. By definition of the operator norm,
\begin{equation*}
\|A(x - \overline{x})\| \leq \|A\|_{\mathrm{op}}  \|x - \overline{x}\| ~ \Longrightarrow ~ \|Ax - \overline{y}\|^2 \leq \|A\|_{\mathrm{op}}^2  \|x - \overline{x}\|^2.
\end{equation*}
Substituting this into \eqref{eq:quasar-op} and using $g(\overline{x}) = h(\overline{y})$ and $g(x) = h(Ax)$, we obtain
\begin{align*}
g\bigl(\lambda \overline{x} + (1-\lambda)x\bigr) 
&\leq \gamma\lambda  g(\overline{x}) + (1 - \gamma\lambda)  g(x) + \tfrac{\gamma\rho}{2}  \lambda\left(1 - \tfrac{\lambda}{2-\gamma}\right) \|A\|_{\mathrm{op}}^2  \|x - \overline{x}\|^2 \\
&= \gamma\lambda  g(\overline{x}) + (1 - \gamma\lambda)  g(x) 
   + \tfrac{\gamma \bigl(\rho\|A\|_{\mathrm{op}}^2\bigr)}{2}  \lambda\left(1 - \tfrac{\lambda}{2-\gamma}\right) \|x - \overline{x}\|^2.
\end{align*}
Hence $g$ is $(\rho\|A\|_{\mathrm{op}}^2, \gamma)$-weak quasar-convex at $\overline{x}$.
\end{proof}
%------------end proof--------------%

Another interesting property of weakly quasar-convex functions, is the following extension of \cite[Theorem 4.8]{BLV} from the quasar-convex to the weakly quasar-convex case, (see also \cite[Theorem 3]{HL-sepa}). Its similar proof is omitted.% \abbas{this sentence related to? the following have a full proof}

%________________________%
%       proposition      %
%________________________%
\begin{proposition} \label{theo:quasar-sepa}
Let $h_{i}: \R^{n_{i}} \rightarrow \R$, with $i \in \{1, \cdots, m\}$ be functions, with $n_{i} \geq 0$ and $\sum^{m}_{i=1} n_{i}=n$, and $h:\R^{n} \rightarrow \R$ be a function such that
\begin{equation}\label{eq:quasar-separable}
 h(x_{1}, \cdots, x_{m})=\sum_{i=1}^{m} h_{i}(x_{i}),
\end{equation}
Let $\overline{x} = (\overline{x}_{1}, \cdots ,\overline{x}_{m}) \in \R^{n}$. Then, for any $\rho \geq 0$ and $\gamma \in ]0,1]$, $h$ is $(\rho, \gamma)$-weakly quasar-convex with respect to $\overline{x} \in \argmin_{\R^{n}}~h$ if and only if every $h_i$ is $(\rho,\gamma)$-weakly quasar-convex with respect to $\overline{x}_{i} \in \argmin_{ \R^{n_{i}}} h_{i}$ (with the \emph{same} $\rho$ and $\gamma$).
\end{proposition}
%------------proof--------------%

The proof of the following characterization follows
the reasoning line from \cite[Theorem 31]{ahookhosh2026quasar} (which studies the strong case), as a consequence, it is omitted.

%________________________%
%       proposition      %
%________________________%
\begin{proposition}[Clarke First-Order Characterization] \label{theo:char}
 Assume $h$ is a locally Lipschitz function with a convex domain. Then $h$ is $(\rho, \gamma)$-weakly quasar-convex with respect to $\overline{x} \in \argmin h$, where $\rho \geq 0$ and $\gamma \in   ]0, 1]$, if and only if for every $x \neq \overline{x}$, we have
\begin{equation}\label{Clarke:char}
 h(\overline{x}) \geq h(x) + \tfrac{1}{\gamma} \langle v, \overline{x} - x \rangle - \tfrac{\rho}{2} \|x - \overline{x}\|^2, ~ \forall  v \in \partial^{C} h(x).
\end{equation}
\end{proposition}
%------------------%

Since a continuously differentiable function is locally Lipschitz, and its Clarke subdifferential reduces to the gradient (i.e., $\partial^C h(x) = \{\nabla h(x)\}$  \cite{Clarke1990}), Proposition \ref{theo:char} immediately yields the following smooth characterization. 

%________________________%
%       corollary        %
%________________________%
\begin{corollary}[First-Order Characterization]\label{cor:star-first_order}
 Assume that $h$ is continuously differentiable on a convex domain $\dom(h)$. Then $h$ is $(\rho, \gamma)$-weakly quasar-convex with respect to $\overline{x} \in \argmin h$, where $\rho \geq 0$ and $\gamma \in   ]0, 1]$, if and only if 
\begin{equation*}
 h(\overline{x})  \geq h(x) + \tfrac{1}{\gamma} \langle \nabla h(x), \overline{x} - x \rangle - \tfrac{\rho}{2} \|x - \overline{x}\|^2.
\end{equation*}
\end{corollary}
%------------------%

An immediate structural consequence of combining the $\mu$-local quadratic growth condition~\eqref{eq:QG} with $(\rho, \gamma)$-weak quasar-convexity under the assumption $\mu > \rho$, is that the $\zeta$-sublevel set $\mathcal{L}_{\zeta}^h(\mathbb{R}^n, \overline{x})$ is free of spurious stationary points, as we show next.

%________________________%
%     proposition     *  %
%________________________%
\begin{proposition}[Absence of Spurious Stationary Points]\label{prop:no_spurious_stationary_points}
 Let $h: \mathbb{R}^n \to \mathbb{R}$ be a continuously differentiable function that satisfies the $\mu$-local quadratic growth condition \eqref{eq:QG} and is $(\rho,\gamma)$-weakly quasar-convex with respect to the global minimizer $\overline{x}$, with $\rho\geq 0$ and $\gamma \in ]0,1]$. If $\mu > \rho$, then $h$ contains no spurious stationary points within the sublevel set $\mathcal{L}_{\zeta}^h (\mathbb{R}^n, \overline{x})$ for some $\zeta>0$, that is, for every $x \in \mathcal{L}_{\zeta}^h (\mathbb{R}^n, \overline{x})$ with $x \neq \overline{x}$, we have $\nabla h(x) \neq 0$.
\end{proposition}
%-----------
\begin{proof}
Let $x \in \mathcal{L}_{\zeta}^h (\mathbb{R}^n, \overline{x})$. By the first-order characterization in Corollary~\ref{cor:star-first_order}, we have (the fundamental gradient bound)
\begin{equation}\label{eq:grad_ineq_wqc}
    \langle \nabla h(x), x - \overline{x} \rangle \geq \gamma(h(x) - h(\overline{x})) - \tfrac{\gamma\rho}{2} \|x - \overline{x}\|^2.
\end{equation}
Since $x \in \mathcal{L}_{\zeta}^h (\mathbb{R}^n,\overline{x})$, the local quadratic growth condition \eqref{eq:QG} applies, giving $h(x) - h(\overline{x}) \geq \tfrac{\mu}{2}\|x - \overline{x}\|^2$. Because $\gamma > 0$, we can substitute this lower bound into \eqref{eq:grad_ineq_wqc} to obtain
\begin{equation*}
    \langle \nabla h(x), x - \overline{x} \rangle \geq \gamma\left(\tfrac{\mu}{2}\|x - \overline{x}\|^2\right) - \tfrac{\gamma\rho}{2} \|x - \overline{x}\|^2 = \tfrac{\gamma(\mu - \rho)}{2} \|x - \overline{x}\|^2.
\end{equation*}
Now, suppose for the contrary that there exists a spurious stationary point $x \in \mathcal{L}_{\zeta}^h (\mathbb{R}^n,\overline{x})$ with $x \neq \overline{x}$. Hence, $\nabla h(x) = 0$. Also, since $\gamma > 0$ and $\mu > \rho$ by assumption, we get $\gamma(\mu - \rho) > 0$, and since $x \neq \overline{x}$, $\|x - \overline{x}\|^2 > 0$.  Consequently,
\begin{equation*}
    0 \geq \tfrac{\gamma(\mu - \rho)}{2} \|x - \overline{x}\|^2 > 0,
\end{equation*}
a contradiction. Therefore, no such spurious stationary point can exist in $\mathcal{L}_{\zeta}^h (\mathbb{R}^n,\overline{x})$.
\end{proof}
%------------end proof--------------%

%________________________%
%       Remark           %
%________________________%
\begin{remark}[Alternative Sufficient Conditions for Local QG] \label{rem:alternative_sufficient_QG}
 Several other well-known regularity conditions guarantee the local quadratic growth condition \eqref{eq:QG}. Notably, \emph{local strong (quasar-)convexity} inherently enforces a quadratic lower bound, directly satisfying the QG condition. Furthermore, the \emph {Restricted Secant Inequality (RSI)} ensures that the gradient points sufficiently toward the optimal set; as established in the literature, strong quasar-convexity and strong star quasi-convexity imply the RSI, which is fundamentally linked to and often used to derive quadratic growth \cite[Properties 6.13, 6.14, 6.17, 6.18]{khademi2026convexity}. These conditions collectively form a hierarchy of regularity that relaxes classical strong convexity while preserving the favorable algorithmic guarantees provided by local QG.
 \qed
\end{remark}
%------------------------

In the following result, we provide a sufficient condition for weakly quasar-convex functions to satisfies the RSI property locally around the minimizer.
%________________________%
%       proposition      %
%________________________%
\begin{proposition} \label{prop:LQG_WQC_implies_RSI}
 Let $h: \mathbb{R}^n \to \mathbb{R}$ be a continuously differentiable and $(\rho,\gamma)$-weakly quasar-convex function with respect to $\overline{x} \in \argmin h$, with $\rho\geq 0$ and $\gamma \in ]0,1]$. If $h$ satisfies the $\mu$-local quadratic growth condition \eqref{eq:QG} on $\mathcal{L}_{\zeta}^h(\mathbb{R}^n, \overline{x})$ for some $\zeta > 0$ and $\mu > \rho$, then $h$ satisfies RSI on $\mathcal{L}_{\zeta}^h(\mathbb{R}^n, \overline{x})$ with parameter $\nu = \tfrac{\gamma (\mu - \rho)}{2} > 0$, that is,
\begin{equation}\label{eq:RSI_derived}
 \langle \nabla h(x), x - \overline{x} \rangle \geq \tfrac{\gamma (\mu - \rho)}{2} \|x - \overline{x} \|^2, ~ \forall  x \in \mathcal{L}_{\zeta}^h (\mathbb{R}^n, \overline{x}).
\end{equation}
\end{proposition}

%------------proof--------------%
\begin{proof}
 By the first-order characterization of $(\rho, \gamma)$-weakly quasar-convexity (Corollary \ref{cor:star-first_order}), we have 
\begin{equation*}
 \langle \nabla h(x), x - \overline{x} \rangle \geq \gamma \bigl( h(x) - h(\overline{x}) \bigr) - \tfrac{\gamma \rho}{2} \|x - \overline{x}\|^2, ~ \forall  x \in \mathbb{R}^{n}.
\end{equation*}
Since the local quadratic growth condition \eqref{eq:QG} holds, there exists $\zeta > 0$ such that for all $x \in \mathcal{L}_{\zeta}^h(\mathbb{R}^n, \overline{x})$ we have $h(x) - h(\overline{x}) \geq \tfrac{\mu}{2}\|x - \overline{x}\|^2$. Using this in the previous inequality and since $\gamma > 0$, we obtain
\begin{equation*}
    \langle \nabla h(x), x - \overline{x} \rangle \geq \gamma\left(\tfrac{\mu}{2}\|x - \overline{x}\|^2\right) - \tfrac{\gamma\rho}{2} \|x - \overline{x}\|^2 = \tfrac{\gamma(\mu - \rho)}{2} \|x - \overline{x}\|^2,
\end{equation*}
which completes the proof.
\end{proof}
%------------end proof--------------%

%%==========================%%
%%       Moreau Envelope    %%
%%==========================%%
\subsection{Moreau Envelope}

We recall the following well-known property of the Moreau envelope (e.g., see \cite{KabganiAhookhosh2026}), and then show that this smoothing operation also preserves weak quasar-convexity.
%________________________%
%       Lemma            %
%________________________%
\begin{lemma}\label{lem:moreau-argmin}
 Let $h$ be proper lower semicontinuous function, and $\overline{x} \in \argmin_{\mathbb{R}^n} h$. Then, for every ${\beta}>0$, $h^{\beta}$ is finite-valued on $\mathbb{R}^n$, $h^{\beta}(\overline{x})=h(\overline{x})$, and $\argmin_{\mathbb{R}^n} h^{\beta} = \argmin_{\R^n} h$.
\end{lemma}
%-----------------------

We now show that the Moreau envelope preserves the weak quasar-convexity of the original function.

%________________________%
%       Theorem          %
%________________________%
\begin{theorem}[Moreau Envelope Preserves Weak Quasar-Convexity] \label{prop:moreau-structural}
Let $h: \R^n \to \overline{\mathbb{R}}$ be proper, lower semicontinuous, and $(\rho, \gamma)$-weakly quasar-convex with respect to $\overline{x} \in \argmin_{\R^n} h$, with $\rho\geq 0$ and $\gamma \in ]0,1]$. Let 
\begin{equation}\label{param:cond:2}
 0 < \beta < \tfrac{2-\gamma}{\rho\gamma},
\end{equation} 
with the convention that the right-hand side is $+\infty$ when $\rho = 0$. Then $h^{\beta}$ is $(\rho_{\beta}, \gamma)$-weakly quasar-convex with respect to the same point $\overline{x} \in \argmin_{\R^n} h^{\beta}$, where
\begin{equation}\label{eq:rho-mu}
 \rho_{\beta} = \frac{(2-\gamma)\rho}{(2-\gamma) - \gamma \rho {\beta}} \ge 0 .
\end{equation}
\end{theorem}
%-------------
\begin{proof}
By Lemma~\ref{lem:moreau-argmin} (applied with this particular ${\beta}>0$), $h^{\beta}$ is finite-valued, $h^{\beta}(\overline{x})=h(\overline{x})$, and $\overline{x}\in\argmin_{\R^n} h^{\beta}$. It remains to prove the chord inequality for $h^{\beta}$ with modulus $\rho_{\beta}$.
Fix $x\in\R^n$, $\lambda\in[0,1]$, and $\varepsilon>0$. Choose $y^\ast\in\R^n$ with
\begin{equation}\label{eq:near-opt}
h(y^\ast)+\tfrac{1}{2{\beta}}\|x-y^\ast\|^2 \leq h^{\beta}(x)+\varepsilon .
\end{equation}
Set $z_\lambda := \lambda\overline{x}+(1-\lambda)x$ and $y_\lambda := \lambda\overline{x}+(1-\lambda)y^\ast$.
Since $y_\lambda$ is feasible for $h^{\beta}(z_\lambda)$ and $z_\lambda-y_\lambda=(1-\lambda)(x-y^\ast)$,
\begin{equation}\label{eq:step21}
h^{\beta}(z_\lambda) \leq h(y_\lambda) + \tfrac{(1-\lambda)^2}{2{\beta}}\|x-y^\ast\|^2 .
\end{equation}
Applying \eqref{eq:quasar-wconvex} to the pair $(\overline{x},y^\ast)$, we have
$$
h(y_\lambda) \leq \gamma\lambda h(\overline{x}) + (1-\gamma\lambda)h(y^\ast) + \tfrac{\gamma\rho}{2}\lambda\Big(1-\tfrac{\lambda}{2-\gamma}\Big)\|y^\ast-\overline{x}\|^2 .
$$
Substituting into \eqref{eq:step21} and using $h(\overline{x})=h^{\beta}(\overline{x})$,
\begin{equation}\label{eq:step31}
h^{\beta}(z_\lambda) \leq \gamma\lambda h^{\beta}(\overline{x}) + (1-\gamma\lambda)h(y^\ast) + \tfrac{\gamma\rho}{2}\lambda\Big(1-\tfrac{\lambda}{2-\gamma}\Big)\|y^\ast-\overline{x}\|^2 + \tfrac{(1-\lambda)^2}{2{\beta}}\|x-y^\ast\|^2 .
\end{equation}
From \eqref{eq:near-opt}, $h(y^\ast)\leq h^{\beta}(x)+\varepsilon-\tfrac{1}{2{\beta}}\|x-y^\ast\|^2$. Since $1-\gamma\lambda\ge0$, substitute into \eqref{eq:step31} we get
\begin{align}
h^{\beta}(z_\lambda) &\leq \gamma\lambda h^{\beta}(\overline{x})+(1-\gamma\lambda)\big(h^{\beta}(x)+\varepsilon\big) + \tfrac{\gamma\rho}{2}\lambda\Big(1-\tfrac{\lambda}{2-\gamma}\Big)\|y^\ast-\overline{x}\|^2 \nonumber\\
&\qquad + \tfrac{1}{2{\beta}}\big((1-\lambda)^2-(1-\gamma\lambda)\big)\|x-y^\ast\|^2 . \label{eq:step41}
\end{align}
Since $(1-\lambda)^2-(1-\gamma\lambda) = -\lambda(2-\gamma-\lambda)$, $\lambda\in[0,1]$ and $\gamma \in   ]0,1]$, we have $2-\gamma-\lambda\ge1-\gamma\ge0$, this bracket is $\le0$. Let $c := \tfrac{ \gamma\rho}{2-\gamma}\ge0$, so $\tfrac{\gamma\rho}{2}\big(1-\tfrac{\lambda}{2-\gamma}\big)=\tfrac{c}{2}(2-\gamma-\lambda)$. Factor $\tfrac{\lambda(2-\gamma-\lambda)}{2}$ from the two quadratic terms in \eqref{eq:step41} gives
$$ \tfrac{c}{2}(2-\gamma-\lambda)\lambda\|y^\ast-\overline{x}\|^2 - \tfrac{\lambda(2-\gamma-\lambda)}{2{\beta}}\|x-y^\ast\|^2 = \tfrac{\lambda(2-\gamma-\lambda)}{2}\Big(c\|y^\ast-\overline{x}\|^2 - \tfrac1{\beta}\|x-y^\ast\|^2\Big).$$
Set $u:=y^\ast-x$ and $v:=x-\overline{x}$, so $y^\ast-\overline{x} = u+v$, and
$$ c \|y^\ast-\overline{x}\|^2-\tfrac1{\beta}\|x-y^\ast \|^2 = -\Big(\tfrac1{\beta}-c\Big)\|u\|^2 + 2 c\langle u,v\rangle + c\|v\|^2.$$
Let $A := \tfrac1{\beta} - c$. The assumption \eqref{param:cond:2} is equivalent to $c\beta<1$. Indeed, if $\rho=0$, then $c=0$, and if $\rho>0$, then $\beta<\tfrac{2-\gamma}{\gamma\rho}$ if and only if $\tfrac{\gamma\rho\beta}{2-\gamma}<1$. Hence, $A = \tfrac1\beta-c=\tfrac{1-c\beta}{\beta}>0$.

 We can maximize this quadratic form over $u$ by completing the square
\begin{align*}
 - A \|u\|^2 + 2c\langle u, v \rangle + c\|v\|^2 &= -A \left\|u - \tfrac{c}{A} v \right\|^2 + \tfrac{c^2}{A} \|v\|^2 + c\|v\|^2 \leq c \left(\tfrac{c+A}{A} \right)\|v\|^2.
\end{align*}
Substituting $A = \tfrac1{\beta} - c$ back into the scalar coefficient yields
$c\tfrac{(c + (\tfrac{1}{\beta} - c))}{(\tfrac{1}{\beta} - c)} = \tfrac{c}{1 - c {\beta}}$. Then, we obtain the tight upper bound
$$ c\|y^\ast-\overline{x}\|^2-\tfrac1{\beta}\|x-y^\ast\|^2 \leq \tfrac{c}{1-c{\beta}}\|x-\overline{x}\|^2 .$$
Hence,
$$
\tfrac{\gamma\rho}{2}\lambda\Big(1-\tfrac{\lambda}{2-\gamma}\Big)\|y^\ast-\overline{x}\|^2 - \tfrac{\lambda(2-\gamma-\lambda)}{2{\beta}}\|x-y^\ast\|^2 \leq \tfrac{\lambda(2-\gamma-\lambda)}{2}\cdot\tfrac{c}{1-c{\beta}}\|x-\overline{x}\|^2 .
$$
Substituting into \eqref{eq:step41} and using $\lambda(2-\gamma-\lambda)=(2-\gamma)\lambda\big(1-\tfrac{\lambda}{2-\gamma}\big)$, we have
\begin{align*}
 & \hspace{0.8cm} h^{\beta}(z_\lambda) \leq \gamma \lambda h^{\beta}(\overline{x})+(1-\gamma \lambda) h^{\beta}(x) + (1-\gamma \lambda) \varepsilon + \tfrac{c (2-\gamma)}{2 (1-c{\beta})}\lambda \Big(1-\tfrac{\lambda}{2-\gamma} \Big) \|x-\overline{x}\|^2 \\
 & \overset{\varepsilon \downarrow 0}{\Longrightarrow} \, h^{\beta}(z_\lambda) \leq \gamma\lambda h^{\beta}(\overline{x}) + (1-\gamma\lambda)h^{\beta}(x) + \tfrac{c(2-\gamma)}{2(1-c{\beta})}\lambda\Big(1-\tfrac{\lambda}{2-\gamma}\Big)\|x-\overline{x}\|^2 .
\end{align*}
Comparing with \eqref{eq:quasar-wconvex}, this is the $(\rho_{\beta},\gamma)$-weak quasar-convexity inequality for $h^{\beta}$, with
$$
\tfrac{\gamma\rho_{\beta}}{2} = \tfrac{c(2-\gamma)}{2(1-c{\beta})} \iff \rho_{\beta} = \tfrac{(2-\gamma)c}{\gamma(1-c{\beta})} = \tfrac{\rho}{1-c{\beta}},
$$
using $c = \tfrac{\gamma \rho}{2-\gamma}$. This is exactly \eqref{eq:rho-mu}. Since $x\in\R^n$ and $\lambda \in [0,1]$ were arbitrary, the proof is complete.
\end{proof}

%------------end proof--------------%
\begin{remark}\label{rem:beta-range}
Let $\rho\ge 0$ and $\gamma \in ]0,1]$, with the convention that both $\tfrac{1}{\rho}$ and $\tfrac{2-\gamma}{\gamma \rho}$ equal $+\infty$ when $\rho=0$.
Since $\gamma \leq 2-\gamma$, it holds that
\[
 \tfrac{1}{\rho} \leq \tfrac{2-\gamma}{\gamma\rho},
\]
with equality if and only if $\gamma=1$ (or $\rho=0$). Consequently, Theorem~\ref{prop:moreau-structural} applies in particular for every $0< \beta <\tfrac{1}{\rho}$, and the same conclusion remains valid. In this range, the relation $c= \tfrac{\gamma\rho}{2-\gamma} \le \rho$ yields the simpler (yet weaker) upper bound
\[
 \rho_{\beta} = \tfrac{\rho}{1-c\beta} \leq \tfrac{ \rho}{1-\rho\beta}.
\]
Condition \eqref{param:cond:2} is precisely the requirement needed in the proof, namely $A=\tfrac{1}{\beta}-c>0$.
\qed
\end{remark}
%%==========================%%
%%    Proximity Operator    %%
%%==========================%%
\subsection{On the Proximity Operator}

In this subsection, we present the basic properties of the proximity operator of a weak quasar-convex function.

As a first result, we show a sufficient condition for the supercoercivity of the proximity operator of a weakly quasar-convex function.

\begin{lemma}\label{coercivity:2}
Let $h:\R^n\to\overline{\R}$ be proper with $\overline{x}\in\argmin_{\R^n}h$, and let $\beta>0$.
Then, for every $x\in\R^n$, the function
$\Phi(y):=h(y)+\tfrac{1}{2\beta}\|x-y\|^2$ is $2$-supercoercive; more precisely,
\begin{equation}\label{eq:2-supercoercive}
  \Phi(y)\ \ge\ h(\overline{x})+\tfrac{1}{2\beta}\|x-y\|^2\quad\forall y\in\R^n,
  \qquad\text{hence}\qquad
  \liminf_{\|y\|\to\infty}\frac{\Phi(y)}{\|y\|^2}\ \ge\ \frac{1}{2\beta}>0 .
\end{equation}
\end{lemma}

\begin{proof}
The first inequality follows from $h(y)\ge h(\overline{x})$. Dividing by $\|y\|^2$ and using
$\|x-y\|^2/\|y\|^2\to1$ as $\|y\|\to\infty$ gives the second one.
\end{proof}

\begin{remark}
Lemma~\ref{coercivity:2} does not use weak quasar-convexity, nor any restriction on $\beta$:
it only uses that $\overline{x}$ is a global minimizer of $h$. In particular, the proximal
subproblems in Corollary~\ref{prox:nonempty} and Algorithm~\ref{Alg: PPM} are solvable for
every $\beta>0$, and the upper bound $\beta_k<1/\rho$ in \eqref{param:cond-2} can be dropped.
\end{remark}

As a consequence of Lemma~\ref{coercivity:2} and the lower semicontinuity of $h$. 

%________________________%
%       Corollary        %
%________________________%
\begin{corollary}\label{prox:nonempty}
Let $\mathcal{K}\subseteq\R^n$ be a nonempty closed convex set and let $h:\R^n\to\overline{\R}$ be proper and
lower semicontinuous, with $\mathcal{K}\cap\dom(h)\neq\emptyset$ and $\overline{x}\in\argmin_{\mathcal{K}} h$. Then, for every
$\beta>0$ and every $x\in\R^n$, the set $\Prox_{\beta h}(\mathcal{K},x)$ is nonempty and compact. In particular,
this holds for every $(\rho,\gamma)$-weakly quasar-convex function with respect to $\overline{x}$,
without any restriction on $\beta$.
\end{corollary}

\begin{proof}
 For every $y\in \mathcal{K}$ we have $h(y)\ge h(\overline{x})$, hence $h(y)+\tfrac{1}{2\beta} \|x-y\|^2 \ge h(\overline{x})+\tfrac{1}{2\beta}\|x-y\|^2$, exactly as in Lemma \ref{coercivity:2}. Thus $\Phi:=h+\iota_{\mathcal{K}}+\tfrac{1}{2\beta}\|x-\cdot\|^2$ is proper, lower semicontinuous and coercive, so its set of minimizers is nonempty and compact.
\end{proof}

 Note that ${\rm Prox}_{\beta h} (\cdot, x)$ may be set-valued even for (strongly) quasar-convex functions as noted in \cite[Remark 14]{de2025extending}.

%________________________%
%       proposition      %
%________________________%
\begin{proposition} \label{prop:quasar-prox-char}
Let $\mathcal{K} \subseteq \mathbb{R}^{n}$ be a closed convex set and $h: \R^n \to \overline{\R}$ be a proper, lower semicontinuous and $(\rho, \gamma)$-weakly quasar-convex with respect to $\overline{x} \in \argmin_{\mathcal{K}} h$ ($\gamma \in   ]0, 1]$ and $\rho \geq 0$) on $\mathcal{K} \subseteq \dom(h)$. 
Let $z \in \mathcal{K}$ and $\beta > 0$. If $x^* \in \prox_{\beta h} (\mathcal{K}, z)$, then
\begin{equation}\label{eq:quasar-prox-char}
 h(x^*) - h(\overline{x}) \leq \tfrac{1}{\gamma\beta} \langle x^* - z, \overline{x} - x^* \rangle + \tfrac{\rho}{2} \| x^* - \overline{x} \|^{2}.
\end{equation}
\end{proposition}
%------------------%
\begin{proof}
Since $x^* \in \prox_{\beta h} (\mathcal{K}, z)$, we have
\begin{equation}\label{eq:prox-ineq}
 h(x^*) + \tfrac{1}{2\beta} \| x^* - z \|^{2} \leq h(x) + \tfrac{1}{2\beta} \| x - z \|^{2}, ~ \forall  x \in \mathcal{K}.
\end{equation}
Fix $\lambda \in ]0, 1]$ and set $x = \lambda \overline{x} + (1-\lambda)x^* \in \mathcal{K}$. Since $h$ is  $(\rho, \gamma)$-weakly quasar-convex,
{\small\begin{align*}
 h(x^*) + \tfrac{1}{2\beta} \| x^* - z \|^{2} & \leq \gamma\lambda  h(\overline{x}) + (1 - \gamma\lambda)  h(x^*) + \tfrac{\gamma\rho}{2}  \lambda\left(1 - \tfrac{\lambda}{2-\gamma}\right) \| x^* - \overline{x} \|^{2} + \tfrac{1}{2 \beta} \| \lambda \overline{x} + (1-\lambda)x^* - z \|^{2} \\
 & = \gamma\lambda  h(\overline{x}) + (1 - \gamma\lambda)  h(x^*) + \tfrac{\gamma\rho}{2}  \lambda\left(1 - \tfrac{\lambda}{2-\gamma}\right) \| x^* - \overline{x} \|^{2} \\
&\quad + \tfrac{\lambda}{2\beta} \| \overline{x} - z \|^{2} 
   + \tfrac{1-\lambda}{2\beta} \| x^* - z \|^{2} 
   - \tfrac{\lambda(1-\lambda)}{2\beta} \| \overline{x} - x^* \|^{2}.
\end{align*}}
That is,
\begin{align*}
\gamma\lambda \bigl( h(x^*) - h(\overline{x}) \bigr) &\leq \tfrac{\lambda}{2\beta} \bigl( \| \overline{x} - z \|^{2} - \| x^* - z \|^{2} \bigr)  + \tfrac{\gamma\rho}{2}  \lambda\left(1 - \tfrac{\lambda}{2-\gamma}\right) \| x^* - \overline{x} \|^{2} - \tfrac{\lambda(1-\lambda)}{2\beta} \| \overline{x} - x^* \|^{2}.
\end{align*}
Dividing by $\lambda > 0$ and using identity \eqref{id:3points}, we obtain
\begin{align*}
 \gamma \bigl( h(x^*) - h(\overline{x}) \bigr) 
 & \leq \tfrac{1}{\beta} \langle x^* - z, \overline{x} - x^* \rangle + \tfrac{1}{2\beta} \| \overline{x} - x^* \|^{2} + \tfrac{\gamma\rho}{2} \left(1 - \tfrac{\lambda}{2-\gamma}\right) \| x^* - \overline{x} \|^{2} - \tfrac{1-\lambda}{2\beta} \| \overline{x} - x^* \|^{2} \\
 & = \tfrac{1}{\beta} \langle x^* - z, \overline{x} - x^* \rangle + \tfrac{\lambda}{2\beta} \| x^* - \overline{x} \|^{2} + \tfrac{\gamma\rho}{2} \left(1 - \tfrac{\lambda}{2-\gamma}\right) \| x^* - \overline{x} \|^{2},
\end{align*}
and the result follows by taking $\lambda \downarrow 0$.
%\begin{equation*}
%h(x^*) - h(\overline{x}) \leq \tfrac{1}{\gamma\beta} \langle x^* - z, \overline{x} - x^* \rangle + \tfrac{\rho}{2} \| x^* - \overline{x} \|^{2},
%\end{equation*}
%which is precisely \eqref{eq:quasar-prox-char}.
\end{proof}
%------------end proof--------------%

The following sufficient decrease property holds:
%________________________%
%       proposition      %
%________________________%
\begin{proposition} \label{prop:quasar-prox-descent}
 Let $\mathcal{K} \subseteq \mathbb{R}^{n}$ be a closed convex set and $h: \R^n \to \overline{\R}$ be a proper, lower semicontinuous, and $(\rho, \gamma)$-weakly quasar-convex with respect to $\overline{x} \in \argmin_{\mathcal{K}} h$ ($\gamma \in   ]0, 1]$ and $\rho \geq 0$) on $\mathcal{K} \subseteq \dom(h)$. Then for every $z \in \mathcal{K}$, $\beta > 0$ and $x^* \in \prox_{\beta h} (\mathcal{K}, z)$, we have
\begin{equation}\label{eq:quasar-prox-descent}
 h(x^*) + \tfrac{1}{2\beta} \| x^* - z \|^{2} \leq h(z).
\end{equation}
Furthermore, if $0 < \gamma\beta\rho < 1$, then
\begin{equation*}
\| x^* - \overline{x} \| \leq \frac{1}{\sqrt{1 - \gamma\beta\rho}}  \| z - \overline{x} \|.
\end{equation*}
\end{proposition}
%------------------%
\begin{proof}
The inequality \eqref{eq:quasar-prox-descent} holds for any proper functions (does not need convexity assumptions), and follows directly by taking $x = z$ in \eqref{eq:prox-ineq}.

For the second part, from relation \eqref{eq:quasar-prox-char} and since $h(x^*) \geq h(\overline{x})$, we have
%\begin{equation}\label{eq:qpc}
$$0 \leq \tfrac{1}{\gamma\beta} \langle x^* - z, \overline{x} - x^* \rangle + \tfrac{\rho}{2} \| x^* - \overline{x} \|^{2}.$$
%\end{equation}
Using identity \eqref{id:3points} we obtain 
\begin{align*}
 & 0 \leq \tfrac{1}{2\gamma\beta} \left( \| \overline{x} - z \|^{2} - \| x^* - z \|^{2} - \| \overline{x} - x^* \|^{2} \right) + \tfrac{\rho}{2} \| x^* - \overline{x} \|^{2}, \\ 
  \Longrightarrow & \left( \tfrac{1}{\gamma\beta} - \rho \right) \| x^* - \overline{x} \|^{2} \leq \tfrac{1}{\gamma\beta} \left( \| \overline{x} - z \|^{2} - \| x^* - z \|^{2} \right).
\end{align*}
Since $\| x^* - z \|^{2} \geq 0$, the right-hand side is bounded above by $\tfrac{1}{\gamma\beta} \| \overline{x} - z \|^{2}$. Under the assumption $0 < \gamma\beta\rho < 1$ (so that $\tfrac{1}{\gamma\beta} - \rho > 0$), we obtain
\begin{equation*}
\| x^* - \overline{x} \|^{2} \leq \frac{\tfrac{1}{\gamma\beta}}{\tfrac{1}{\gamma\beta} - \rho}  \| \overline{x} - z \|^{2} = \tfrac{1}{1 - \gamma\beta\rho}  \| z - \overline{x} \|^{2},
\end{equation*}
which completes the proof.
\end{proof}
%------------end proof--------------%

%%==========================%%
%%       Algorithm          %%
%%==========================%%
\section{The Proximal Point Method and Convergence Analysis} \label{sec:PPM}

In this part, we consider the following general optimization problem
\begin{equation}\tag{P}\label{P}
 \min\limits_{x \in \mathbb{R}^n} h(x),
\end{equation}
where $h$ is a weakly quasar-convex function. 
Our purpose is to solve problem \eqref{P} via the proximal point method in its classical form (see \cite{rock-1976}).

%

%________________________%
%       Algorithm       %
%________________________%
\begin{algorithm}[hbpt!]
\caption{Proximal Point Method (PPM)}\label{Alg: PPM}

\textbf{Input:}  Initial point $x^0 \in \R^n$, $\beta'>0$, and a sequence of step sizes $\{\beta_k\}_k$ such that
\begin{equation}\label{param:cond-2}
  \beta_k\ \ge\ \beta' > 0, ~ \forall k \in \mathbb{N}_{0}.
\end{equation} 
%Initial point $x^0 \in \mathbb{R}^n$, $\rho \geq 0$, $\beta^{\prime} > 0$, and a sequence of step sizes $\{\beta_k\}_{k}$ such that 
%\begin{equation}\label{param:cond-2}
% 0 < \beta^{\prime} \leq \beta_{k} <  \frac{1}{\rho}, ~ \forall k \in \mathbb{N}_{0}.
%\end{equation} 

\textbf{Iterate:} For $k = 0, 1, 2, \cdots$
\begin{equation}\label{eq:ppm-iteration}
 x^{k+1} \leftarrow \text{Prox}_{\beta_k h}(x^k) = \argmin_{z \in \mathbb{R}^n} \left\{h(z) + \tfrac{1}{2 \beta_k}\|z - x^k\|^2\right\}.
\end{equation}
\textbf{Until:} Terminate when the stopping criteria are satisfied. \\
\textbf{Return:} $x^{k+1}$. 
\end{algorithm}
%------------------%

Corollary \ref{prox:nonempty} establishes the existence of solutions to all proximal subproblems, rendering the iterative scheme well-defined. 

%________________________%
%      lemma             %
%________________________%
\begin{lemma}[Descent Property]\label{prop:ppm-descent}
 Let $h: \mathbb{R}^n \to \overline{\mathbb{R}}$ be proper, lower semicontinuous, and $(\rho,\gamma)$-weakly quasar-convex with respect to $\overline{x} \in \argmin_{ \mathbb{R}^n} h$, with $\rho \geq 0$ and $\gamma \in   ]0, 1]$. Then for every $\beta_k > 0$ and $x^{k+1} \in {\rm Prox}_{\beta_k h} (x^k)$, we have
\begin{equation}\label{eq:descent}
 h(x^{k+1}) + \tfrac{1}{2\beta_k}\|x^{k+1} - x^k\|^2 \leq h(x^k).
\end{equation}
In particular, $\{h(x^k)\}_{k}$ is a non-increasing sequence.
\end{lemma}
%------------------%
\begin{proof}
Directly from \eqref{eq:prox-ineq} with $z=x^{k}$ and $x^{*} = x^{k+1}$.
\end{proof}
%------------------%]

The following useful relation holds.

%________________________%
%       proposition      %
%________________________%
\begin{proposition}\label{Prop: 1}
Let $\mathcal{K} \subset \mathbb{R}^n$ be a closed convex set and $h: \mathbb{R}^n \rightarrow \overline{ \mathbb{R}}$ be a proper, lower semicontinuous and $(\rho, \gamma)$-weakly quasar-convex function with respect to $\overline{x} \in \argmin_{\mathcal{K}} h$ ($\gamma \in   ]0,1]$ and $\rho \geq 0$) on $\mathcal{K} \subset \dom(h)$, $x^k \in \mathbb{R}^n$ and $\beta_{k} > 0$ for all $k \in \mathbb{N}_{0}$. If $x^{k+1} \in {\rm Prox}_{\beta_{k} h}(\mathcal{K}, x^k)$, then
\begin{equation}\label{eq:prop1}
 h\left(x^{k+1}\right) - h(\overline{x}) \leq \tfrac{\rho}{2} \left\|x^{k+1} - \overline{x} \right\|^2 + \tfrac{1}{\gamma \beta_{k}} \langle x^{k+1} - x^k, \overline{x} -x^{k+1} \rangle.
\end{equation}
\end{proposition}
%--------
\begin{proof}
Directly from Proposition \ref{prop:quasar-prox-char} with $z=x^{k}$ and $x^{*} = x^{k+1}$.
\end{proof}
%------------------%

In the next result, we establish the convergence of the sequence generated by Algorithm~\ref{Alg: PPM} to a global solution of the minimization problem.

%________________________%
%     theorem            %
%________________________%
\begin{theorem}[Convergence of PPM under Local Quadratic Growth] \label{thm:ppm-linear}
 Let $h:\mathbb{R}^n\to\overline{\mathbb{R}}$ be proper, lower semicontinuous, $(\rho,\gamma)$-weakly quasar-convex with respect to $\overline{x} \in \argmin_{ \mathbb{R}^n}h$, with $\rho \geq 0$ and $\gamma \in   ]0, 1]$, and that satisfies the $\mu$-local quadratic growth condition \eqref{eq:QG}. Assume that $\mu>\rho$. Let $\{x^k\}_{k}$ be generated by Algorithm \ref{Alg: PPM} with step sizes $\{\beta_{k}\}_{k}$ satisfying \eqref{param:cond-2}. If $x^0 \in \mathcal{L}_{\zeta}^h (\mathbb{R}^n, \overline{x})$ for some $\zeta>0$, then setting
\begin{equation}\label{eq:tau}
 \tau := \frac{1}{1+ (\mu-\rho)\gamma \beta^{\prime}} \in   ]0, 1[,
\end{equation}
we have for every $k \in \mathbb{N}_{0}$ that:
\begin{itemize}
 \item[$(a)$] $x^k\in\mathcal{L}_{\zeta}^h(\mathbb{R}^n,\overline{x})$.
    
 \item[$(b)$] $\|x^{k+1}-\overline{x}\|^2 \leq \tau   \|x^k-\overline{x}\|^2 \leq \tau^{ k+1}\|x^0-\overline{x}\|^2$.

 \item[$(c)$] $h(x^k)\to h(\overline{x})$.
\end{itemize}
\end{theorem}
%------------------%
\begin{proof}
 $(a)$: It follows by Lemma~\ref{prop:ppm-descent} that
\begin{equation}\label{eq:descent-use}
 h(x^{k+1}) \leq h(x^{k+1}) + \tfrac{1}{2\beta_k} \|x^{k+1}-x^k\|^2 \leq h(x^k).
\end{equation}
Hence $\{h(x^k)\}$ is non-increasing, and by induction $h(x^k)\leq h(x^0) \leq h(\overline{x})+\zeta$, so $x^k\in\mathcal{L}_{\zeta}^h(\mathbb{R}^n, \overline{x})$ for every $k \in \mathbb{N}_{0}$.

\noindent
$(b)$: Applying Proposition~\ref{Prop: 1} with $\mathcal{K}=\mathbb{R}^n$ and rearranging yields
\begin{equation}\label{eq:key}
 2 \langle x^{k+1}-x^k,  x^{k+1}-\overline{x} \rangle \leq \rho\gamma\beta_k \|x^{k+1}-\overline{x} \|^2 + 2 \gamma\beta_k (h(\overline{x}) - h(x^{k+1})).
\end{equation}
Since $x^{k+1}\in\mathcal{L}_{\zeta}$ by part $(a)$, condition~\eqref{eq:QG} gives
\begin{equation}\label{eq:QG-applied}
 h(x^{k+1})-h(\overline{x}) \geq \tfrac{\mu}{2} \|x^{k+1}-\overline{x}\|^2,
\end{equation}
so that $h(\overline{x}) - h(x^{k+1}) \leq - \frac{\mu}{2} \|x^{k+1} - \overline{x}\|^2$. Substituting into~\eqref{eq:key},
\begin{equation}\label{eq:key2}
 2 \langle x^{k+1}-x^k,  x^{k+1} - \overline{x} \rangle \leq - (\mu-\rho) \gamma \beta_k \|x^{k+1}-\overline{x}\|^2.
\end{equation}
Using \eqref{id:3points}, we have
\begin{align}
 & \|x^{k+1}-\overline{x}\|^2 - \|x^k-\overline{x}\|^2 + \|x^{k+1}-x^k\|^2 \leq - (\mu-\rho) \gamma \beta_k \|x^{k+1}-\overline{x}\|^2, \notag \\
 & \hspace{1.05cm} \Longrightarrow \bigl( 1 + (\mu-\rho) \gamma \beta_k \bigr) \|x^{k+1} - \overline{x}\|^2 \leq \|x^k - \overline{x}\|^2. \label{eq:contraction}
\end{align}
Since $\mu>\rho$, $\gamma>0$, and $\beta_k \geq \beta^{\prime}>0$, we have $1+(\mu-\rho) \gamma \beta_k \geq 1+(\mu-\rho) \gamma \beta^{\prime} = \tfrac{1}{\tau} >1$, where $\tau \in ]0, 1[$ is defined in~\eqref{eq:tau}. Dividing \eqref{eq:contraction} by $1+(\mu-\rho)\gamma\beta_k$ yields
\begin{equation*}
 \|x^{k+1}-\overline{x}\|^2 \leq \tfrac{1}{1+(\mu-\rho) \gamma \beta_k} \|x^k - \overline{x}\|^2 \leq \tau  \|x^k - \overline{x}\|^2.
\end{equation*}
Applying this bound recursively, we obtain $(b)$.

\noindent
$(c)$: Since $\overline{x}$ is a global minimizer of $h$, $h(x^{k+1}) \geq h(\overline{x})$ for all $k$. Furthermore, by applying Proposition~\ref{Prop: 1} with $\mathcal{K}=\mathbb{R}^n$, we get
\begin{equation}\label{eq:upper-raw}
 0 \leq h(x^{k+1}) - h(\overline{x}) \leq \tfrac{\rho}{2}\|x^{k+1}-\overline{x}\|^2 + \tfrac{1}{\gamma\beta_k} \langle x^{k+1} - x^k, \overline{x} - x^{k+1}\rangle.
\end{equation}
Applying the Cauchy-Schwarz inequality to the inner product and using $\beta_k\geq\beta^{\prime}>0$, we get
\begin{equation}\label{eq:upper}
 h(x^{k+1}) - h(\overline{x}) \leq \tfrac{\rho}{2} \|x^{k+1}-\overline{x}\|^2 + \tfrac{1}{\gamma \beta^{\prime}} \|x^{k+1}-x^k\| \cdot \|x^{k+1}-\overline{x}\|.
\end{equation}
From part $(b)$, we know that $\|x^{k}-\overline{x}\|\leq\tau^{k/2}\|x^0-\overline{x}\|\to 0$. By the triangle inequality,
\begin{equation*}
 \|x^{k+1}-x^k\| \leq \|x^{k+1}-\overline{x}\| + \|x^k-\overline{x}\| \leq \bigl( \tau^{(k+1)/2} + \tau^{k/2}\bigr) \|x^0-\overline{x}\| \to 0.
\end{equation*}
Hence both $\|x^{k+1}-\overline{x}\|\to 0$ and $\|x^{k+1}-x^k\|\to 0$, so the
right-hand side of~\eqref{eq:upper} tends to $0$. Combining with \eqref{eq:upper-raw} yields $h(x^k)\to h(\overline{x})$, which completes the proof.
\end{proof}
%------------end proof--------------%

In particular, we have the following convergence rates.
%________________________%
%       Corollary        %
%________________________%

\begin{corollary}[Convergence Rates] \label{cor:ppm-rates}
Under the same assumptions as Theorem~\ref{thm:ppm-linear}, the following hold:
\begin{itemize}
 \item[$(a)$] The sequence of iterates $\{x^k\}_{k}$ converges $Q$-linearly to $\overline{x}$ with Q-factor $\sqrt{\tau}$, i.e.,
 \begin{equation}
  \|x^{k+1} - \overline{x}\| \leq \frac{1}{\sqrt{1+ (\mu-\rho)\gamma \beta^{\prime}}} \|x^k - \overline{x}\|, ~ \forall  k \in \mathbb{N}_0.
 \end{equation}
 \item[$(b)$] The sequence of function values $\{h(x^k)\}_{k}$ converges R-linearly to $h(\overline{x})$ with R-factor $\sqrt{\tau}$, i.e.,
 \begin{equation}
  h(x^k) - h(\overline{x}) \leq C  (\sqrt{\tau})^{k-1}, ~ \forall k \in \mathbb{N},
 \end{equation}
 where
 \begin{equation}\label{eq:C-def}
  C := \left(\tfrac{\rho}{2} + \tfrac{1}{\gamma\beta'}\big(1+\sqrt{\tau}\big)\right)\|x^0-\overline{x}\|^2.
 \end{equation}
\end{itemize}
\end{corollary}
%------------------
\begin{proof}
$(a)$: This follows directly from Theorem \ref{thm:ppm-linear}$(b)$. The contraction factor $\tau$ comes from the inequality
\begin{equation*}
 \|x^{k+1} - \overline{x}\|^2 \leq \tfrac{1}{1 + (\mu - \rho) \gamma\beta_k}\|x^k - \overline{x}\|^2 \leq \tau \|x^k - \overline{x}\|^2,
\end{equation*}
where the last inequality uses $\beta_k \ge \beta'$ from \eqref{param:cond-2}, and the result follows by applying this recursively.

$(b)$: From Proposition~\ref{prop:quasar-prox-char} with $\mathcal{K}=\mathbb{R}^n$, we have
\begin{equation*} %\label{eq:upper-raw-cor}
 h(x^{k+1}) - h(\overline{x}) \leq \tfrac{\rho}{2} \|x^{k+1} - \overline{x}\|^2 + \tfrac{1}{\gamma \beta_k} \langle x^{k+1} - x^k, \overline{x} - x^{k+1}\rangle.
\end{equation*}
Applying the Cauchy-Schwarz inequality and using $\beta_k \ge \beta' > 0$, we get
\begin{equation*}
 h(x^{k+1}) - h(\overline{x}) \leq \tfrac{\rho}{2} \|x^{k+1}-\overline{x}\|^2 + \tfrac{1}{\gamma \beta'} \|x^{k+1}-x^k\| \cdot \|x^{k+1}-\overline{x}\|.
\end{equation*}
Using part $(a)$ and the triangle inequality,
\begin{equation*}
 \|x^{k+1} - x^k\| \leq \|x^{k+1} - \overline{x}\| + \|x^k - \overline{x}\| \leq \left(\tau^{(k+1)/2} + \tau^{k/2} \right)\|x^0 - \overline{x}\|.
\end{equation*}
Substituting and simplifying,
\begin{align*}
 h(x^{k+1}) - h(\overline{x}) & \leq \left( \tfrac{ \rho}{2} \tau^{k+1} + \tfrac{1}{\gamma \beta'} ( \tau^{(k+1)/2} + \tau^{k/2}) \tau^{ (k+1)/2} \right) \|x^0 - \overline{x}\|^2 \\
 & \leq \left( \tfrac{\rho}{2} + \tfrac{1}{\gamma \beta'} (1 + \sqrt{\tau}) \right) \tau^{k/2} \|x^0 - \overline{x}\|^2 \\
 & = C  \tau^{k/2},
\end{align*}
with $C$ as in \eqref{eq:C-def}. Since this bounds $h(x^{k+1})-h(\overline x)$ for every $k\in \mathbb{N}_0$, reindexing $n:=k+1$ gives
\begin{equation*}
 h(x^n) - h(\overline{x}) \leq C  \tau^{(n-1)/2} = C (\sqrt\tau)^{n-1}, ~ \forall  n \geq 1,
\end{equation*}
which is the claimed bound.
\end{proof}

%------------end proof--------------%
We finish this subsection with the following complexity analysis. The standard proof is omitted.
%________________________%
%       Corollary        %
%________________________%
\begin{corollary}[Iteration Complexity] \label{cor:ppm-complexity}
Under the same assumptions as Corollary~\ref{cor:ppm-rates}, the following assertions hold:
\begin{itemize}
\item[$(a)$] For any $\varepsilon > 0$, after at most
\begin{equation}\label{eq:complexity-iter}
 K_{\mathrm{iter}}(\varepsilon) = \left\lceil \tfrac{2}{- \ln (\tau)} \ln\left(\tfrac{\|x^0 - \overline{x}\|}{\varepsilon} \right) \right\rceil
\end{equation}
iterations, we have $\|x^K - \overline{x}\| \leq \varepsilon$.

\item[$(b)$] For any $\varepsilon > 0$, after at most
\begin{equation}\label{eq:complexity-func}
 K_{\mathrm{func}}(\varepsilon) = \left\lceil 1+ \tfrac{2}{-\ln (\tau)} \ln\left( \tfrac{C}{\varepsilon} \right) \right\rceil
\end{equation}
iterations, we have $h(x^K) - h(\overline{x}) \leq \varepsilon$, where $C$ is defined in Corollary \ref{cor:ppm-rates}.
\end{itemize}
In both cases, the iteration complexity is $O \left(\tfrac{1}{-\ln (\tau)} \ln(1/\varepsilon) \right)$, which is the standard complexity for linearly convergent methods.
\end{corollary}
%--------------------------%
\begin{remark}
 Theorem \ref{thm:ppm-linear} extends \cite[Theorem 4.3]{LDZ}  from weakly convex to weakly quasar-convex functions. Furthermore, when $\gamma = 1$ (i.e., $h$ is weakly star-convex), our convergence rate in Corollary \ref{cor:ppm-rates} is exactly as good as \cite[Theorem 4.3]{LDZ} ($\tau = (1+(\mu-\rho)\beta')^{-1}$ in our case and $\tau= (1+c_{\min}(\mu_q-\rho))^{-1}$ in their case). 
 \qed
\end{remark}

\subsection{Computable Complexity via the Proximal Residual}

The complexity results of Corollaries~\ref{cor:ppm-rates} and~\ref{cor:ppm-complexity} are stated in terms of $\|x^0-\overline{x}\|$ or the constant $C$, both of which require knowledge of the (typically unknown) minimizer $\overline{x}$. In this subsection, we derive a complementary, fully computable stopping criterion based on the \emph{proximal residual}, together with its own iteration complexity, and we show it is consistent with the sub-optimality complexity already established.

%________________________%
%       Definition        %
%________________________%
\begin{definition}[Proximal Residual]\label{def:prox-residual}
 Let $h:\mathbb{R}^n\to\overline{\mathbb{R}}$ be proper, lower semicontinuous, and $\beta>0$. For $x \in \mathbb{R}^n$ and $x^{*} \in \prox_{\beta h}(x)$, the proximal residual at $x$ (associated with $x^*$) is defined by
\begin{equation*}
 r_{\beta}(x) := \tfrac{1}{\beta}\left(x - x^{*} \right).
\end{equation*}
\end{definition}

Unlike $\|x-\overline{x}\|$, the quantity $\| r_{\beta} (x)\|$ is observable during the execution of Algorithm \ref{Alg: PPM}, since along the PPM trajectory $x^{k+1} \in \prox_{\beta_k h}(x^k)$ we have simply $r_{\beta_k}(x^k) = \tfrac{1}{\beta_k}(x^k - x^{k+1})$, computable from two consecutive iterates without any knowledge of $\overline{x}$.

We first relate the residual to sub-optimality, under the local quadratic growth condition \eqref{eq:QG} already used in Corollary~\ref{cor:ppm-rates}.

%________________________%
%       Proposition       %
%________________________%
\begin{proposition}[Proximal-PL-type Inequality]\label{prop:prox-PL}
 Let $h: \mathbb{R}^n \to \overline{\mathbb{R}}$ be proper, lower semicontinuous, $(\rho, \gamma)$-weakly quasar-convex with respect to $\overline{x} \in \argmin_{\mathbb{R}^n} h$, and satisfy the $\mu$-quadratic growth condition \eqref{eq:QG} at $\overline{x}$ globally on $\mathbb{R}^n$ (i.e., for every $\zeta>0$), with $\mu > \rho$. Let $\beta > 0$, $x \in \mathbb{R}^n$, and $x^{*} \in \prox_{\beta h} (x)$. Then,
\begin{equation}\label{eq:prox-PL}
 h(x^{*}) - h(\overline{x})   \leq   \tfrac{2\mu}{\gamma^{2}(\mu-\rho)^{2}} \|r_{\beta}(x)\|^{2}.
\end{equation}
\end{proposition}

%------------------%
\begin{proof}
Taking $z=x$ in Proposition \ref{prop:quasar-prox-char}, we obtain
\begin{equation*}
 h(x^{*}) - h(\overline{x})   \leq   \tfrac{1}{\gamma\beta}\langle x^{*}-x, \overline{x}-x^{*}\rangle + \tfrac{\rho}{2} \|x^{*}-\overline{x}\|^{2}.
\end{equation*}
By definition, $x^{*}-x = -\beta  r_{\beta}(x)$, so, applying the Cauchy-Schwarz inequality,
\begin{equation*}
 \tfrac{1}{\gamma\beta}\langle x^{*}-x, \overline{x}-x^{*}\rangle = -\tfrac{1}{\gamma} \langle r_{\beta}(x), \overline{x}-x^{*} \rangle   \leq   \tfrac{1}{\gamma} \|r_{\beta}(x)\| \|x^{*}-\overline{x}\|.
\end{equation*}
Combining the two displays,
\begin{equation}\label{eq:prox-PL-step1}
 h(x^{*}) - h(\overline{x})   \leq   \tfrac{1}{\gamma} \|r_{\beta}(x)\| \|x^{*}-\overline{x}\| + \tfrac{\rho}{2} \|x^{*}-\overline{x}\|^{2}.
\end{equation}
On the other hand, by the $\mu$-local quadratic growth condition \eqref{eq:QG} applied at $x^{*}$, we have $h(x^{*}) - h(\overline{x})   \geq   \tfrac{\mu}{2}\|x^{*}-\overline{x}\|^{2}$. Combining this lower bound with \eqref{eq:prox-PL-step1} gives
\begin{equation*}
 \tfrac{\mu}{2} \|x^{*}- \overline{x}\|^{2}   \leq   \tfrac{1}{\gamma} \|r_{\beta}(x)\| \|x^{*}-\overline{x}\| + \tfrac{\rho}{2} \|x^{*}-\overline{x} \|^{2}.
\end{equation*}
If $x^{*}= \overline{x}$, then \eqref{eq:prox-PL} holds trivially since $h(x^*) - h(\overline x) = 0$. Otherwise, dividing by $\|x^{*}- \overline{x}\| >0$ and rearranging, using $\mu > \rho$, we get $\|x^{*} - \overline{x}\|   \leq   \tfrac{2}{\gamma (\mu - \rho)} \|r_{\beta}(x)\|$. Substituting this bound back into \eqref{eq:prox-PL-step1},
\begin{align*}
 h(x^{*}) - h(\overline{x}) & \leq \tfrac{1}{\gamma} \|r_{\beta}(x)\| \tfrac{2}{ \gamma (\mu-\rho)} \|r_{\beta}(x)\| + \tfrac{\rho}{2} \left(\tfrac{2}{\gamma (\mu-\rho)} \right)^{2} \|r_{\beta}(x)\|^{2} \\
 & = \tfrac{2}{\gamma^{2} (\mu-\rho)} \|r_{\beta}(x) \|^{2} + \tfrac{2 \rho}{\gamma^{2} (\mu-\rho)^{2}} \|r_{\beta}(x) \|^{2} \\
 & = \tfrac{2 \mu}{\gamma^{2} (\mu-\rho)^{2}} \| r_{\beta} (x)\|^{2},
\end{align*}
which is \eqref{eq:prox-PL}.
\end{proof}
%------------------%

We now show that the PPM iterates drive the residual to zero linearly, at the same rate $\sqrt{\tau}$ already established for the iterate distance and function values.

%________________________%
%       Corollary        %
%________________________%
\begin{corollary}[Residual Convergence Rate] \label{cor:ppm-residual-rate}
 Under the same assumptions as Corollary \ref{cor:ppm-rates}, let $\{x^k\}_{k}$ be generated by Algorithm \ref{Alg: PPM} with $\beta_k \equiv \beta'$ satisfying \eqref{param:cond-2}, and set $$r^{k} := r_{\beta'}(x^{k}) = \tfrac{1}{\beta'} (x^k - x^{k+1}).$$ 
 Then,
\begin{equation}\label{eq:residual-rate}
 \|r^{k}\|   \leq   \tfrac{1+\sqrt{\tau}}{\beta'}  (\sqrt{\tau})^{k}   \|x^0-\overline{x}\|, \qquad \forall   k \in\mathbb{N}_0.
\end{equation}
\end{corollary}
%------------------%
\begin{proof}
By the triangle inequality and Corollary~\ref{cor:ppm-rates}$(a)$,
\begin{equation*}
 \|r^{k}\| = \tfrac{1}{\beta'}\|x^{k}-x^{k+1}\| \leq \tfrac{1}{\beta'} \Big(\|x^{k}-\overline{x}\| + \|x^{k+1} - \overline{x}\| \Big) \leq \tfrac{1}{\beta'} (\tau^{k/2} + \tau^{(k+1)/2}) \|x^{0}-\overline{x}\|,
\end{equation*}
which simplifies to \eqref{eq:residual-rate}.
\end{proof}

%------------------%
%________________________%
%       Corollary        %
%________________________%
\begin{corollary}[Residual-Based Iteration Complexity]\label{cor:ppm-residual-complexity}
 Under the same assumptions as Corollary \ref{cor:ppm-rates}, let $\{x^k\}_{k}$ be generated by Algorithm \ref{Alg: PPM} with $\beta_k \equiv \beta'$ satisfying \eqref{param:cond-2}. Then for any $\varepsilon > 0$, after at most
\begin{equation}\label{eq:complexity-residual}
 K_{\mathrm{res}}(\varepsilon) = \left\lceil \tfrac{2}{-\ln (\tau)}\ln\left(\tfrac{(1+\sqrt{\tau}) \|x^0-\overline{x}\|}{\beta'\varepsilon} \right) \right\rceil
\end{equation}
iterations, we have $\|r^{K}\| \leq \varepsilon$; this criterion is computable online, without knowledge of $\overline{x}$, from $\|x^K-x^{K+1}\|$ alone. Moreover, if the assumptions of Proposition~\ref{prop:prox-PL} hold, then $\|r^K\|\leq\varepsilon$ implies
\begin{equation*}
 h(x^{K+1}) - h(\overline{x})   \leq   \tfrac{2\mu}{\gamma^{2}(\mu-\rho)^{2}} \varepsilon^{2}.
\end{equation*}
In both cases, the iteration complexity is $O\left(\tfrac{1}{-\ln (\tau)}\ln(1/\varepsilon)\right)$, matching the order already established in Corollary~\ref{cor:ppm-complexity}.
\end{corollary}
%------------------%
\begin{proof}
 The bound \eqref{eq:complexity-residual} follows from \eqref{eq:residual-rate}, requiring $\tfrac{1 + \sqrt\tau}{\beta'}\tau^{K/2}\|x^0-\overline x\| \leq \varepsilon$, and solving for $K$ exactly as in the derivation of Corollary \ref{cor:ppm-complexity}. The observability of $\|r^K\| = \tfrac1{\beta'} \|x^K - x^{K+1}\|$ from consecutive iterates is immediate from Definition \ref{def:prox-residual}. The final sub-optimality bound follows directly from Proposition \ref{prop:prox-PL}, applied at $x=x^{K}$, $x^{*}=x^{K+1}$.
\end{proof}

%------------------%
\begin{remark}
 Corollary \ref{cor:ppm-residual-complexity} recovers a complexity of the same order $O \big( \tfrac{1}{-\ln (\tau)}\ln(1/\varepsilon)\big)$ as Corollary \ref{cor:ppm-complexity}, confirming consistency between the two criteria, but with an essential practical advantage: the stopping rule $\|x^K-x^{K+1}\| \leq \beta' \varepsilon$ can be checked by the algorithm itself during execution, whereas the criteria of Corollary \ref{cor:ppm-complexity} require an a priori estimate of $\|x^0-\overline{x}\|$ or of the constant $C$, both of which depend on the unknown minimizer $\overline{x}$.
 \qed
\end{remark}

The following lemma guarantees that both the weak quasar-convexity and the local quadratic growth condition are preserved under such a quadratic perturbation.
\begin{lemma}\label{lem:quadratic-perturbation}
Let $\phi: \mathbb{R}^n \to \overline{\mathbb{R}}$ be $(\rho, \gamma)$-weakly quasar-convex with respect to $\overline{x} \in \argmin_{\mathbb{R}^n} \phi$, and suppose $\phi$ satisfies the $\mu$-local Quadratic Growth (QG) condition at $\overline{x}$ with parameter $\zeta > 0$. Then, for any $L \geq 0$, the perturbed function $\psi(x) := \phi(x) + L\|x - \overline{x}\|^2$ is $(\rho_L, \gamma)$-weakly quasar-convex with respect to $\overline{x}$, and satisfies the $\mu_L$-local QG condition at $\overline{x}$ with the same parameter $\zeta$, where
\begin{equation*}
    \rho_L = \max\left\{0,   \rho - 2L\left(\frac{2}{\gamma} - 1\right)\right\} \quad \text{and} \quad \mu_L = \mu + 2L.
\end{equation*}
\end{lemma}

\begin{proof}
First, observe that since $\overline{x} \in \argmin_{\mathbb{R}^n} \phi$ and $L \geq 0$, we have $\psi(x) = \phi(x) + L\|x - \overline{x}\|^2 \geq \phi(\overline{x}) = \psi(\overline{x})$ for all $x \in \mathbb{R}^n$. This ensures that $\overline{x} \in \argmin_{\mathbb{R}^n} \psi$, which is a strict requirement for the definition of weak quasar-convexity. We now prove the two properties separately. For any $x \in \mathcal{L}_{\zeta}^{\psi}(\mathbb{R}^n, \overline{x})$, we have 
$$ \phi(x) \leq \psi(x) \leq \psi(\overline{x}) + \zeta = \phi(\overline{x}) + \zeta, $$ 
which implies $x \in \mathcal{L}_{\zeta}^{\phi}(\mathbb{R}^n, \overline{x})$. Therefore, the $\mu$-QG condition for $\phi$ applies to $x$. By the definition of $\psi$, we obtain
\begin{align*}
 \psi(x) - \psi(\overline{x}) &= \big(\phi(x) + L\|x - \overline{x}\|^2\big) - \big(\phi(\overline{x}) + 0\big) = \big(\phi(x) - \phi(\overline{x})\big) + L\|x - \overline{x}\|^2 \\
 & \geq \frac{\mu}{2}\|x - \overline{x}\|^2 + L\|x - \overline{x}\|^2 = \frac{\mu + 2L}{2}\|x - \overline{x}\|^2.
\end{align*}
Thus, $\psi$ satisfies the $\mu_L$-QG condition with $\mu_L = \mu + 2L$ on $\mathcal{L}_{\zeta}^{\psi}(\mathbb{R}^n, \overline{x})$.

Now, let $x_\lambda = \lambda \overline{x} + (1-\lambda)x$ for $\lambda \in [0,1]$. Note that $x_\lambda - \overline{x} = (1-\lambda)(x - \overline{x})$, which implies $\|x_\lambda - \overline{x}\|^2 = (1-\lambda)^2\|x - \overline{x}\|^2$. 
Using the definition of $\psi$ and the $(\rho, \gamma)$-weak quasar-convexity of $\phi$, we obtain
\begin{align*}
    \psi(x_\lambda) &= \phi(x_\lambda) + L\|x_\lambda - \overline{x}\|^2 \\
    &\leq \gamma \lambda \phi(\overline{x}) + (1-\gamma \lambda) \phi(x) + \frac{\gamma \rho}{2} \lambda \left(1 - \frac{\lambda}{2-\gamma}\right)\|x - \overline{x}\|^2 + L(1-\lambda)^2\|x - \overline{x}\|^2.
\end{align*}
Substitute $\phi(\overline{x}) = \psi(\overline{x})$ and $\phi(x) = \psi(x) - L\|x - \overline{x}\|^2$ into the inequality
\begin{align*}
    \psi(x_\lambda) &\leq \gamma \lambda \psi(\overline{x}) + (1-\gamma \lambda) \big(\psi(x) - L\|x - \overline{x}\|^2\big)  + \left[ \frac{\gamma \rho}{2} \lambda \left(1 - \frac{\lambda}{2-\gamma}\right) + L(1-\lambda)^2 \right] \|x - \overline{x}\|^2 \\
    &= \gamma \lambda \psi(\overline{x}) + (1-\gamma \lambda) \psi(x) + B(\lambda)\|x - \overline{x}\|^2,
\end{align*}
where $B(\lambda)$ is the coefficient of $\|x - \overline{x}\|^2$. We simplify $B(\lambda)$ by expanding the terms involving $L$
\begin{align*}
    B(\lambda) &= \frac{\gamma \rho}{2} \lambda \left(1 - \frac{\lambda}{2-\gamma}\right) + L \big( 1 - 2\lambda + \lambda^2 - 1 + \gamma \lambda \big) \\
%    &= \frac{\gamma \rho}{2} \lambda \left(1 - \frac{\lambda}{2-\gamma}\right) + L \big( \lambda^2 - (2-\gamma)\lambda \big) \\
    &= \frac{\gamma \rho}{2} \lambda \left(1 - \frac{\lambda}{2-\gamma}\right) - L \lambda (2-\gamma) \left( 1 - \frac{\lambda}{2-\gamma} \right) \\
%    &= \frac{\gamma}{2} \left[ \rho - \frac{2L(2-\gamma)}{\gamma} \right] \lambda \left( 1 - \frac{\lambda}{2-\gamma} \right) \\
    &= \frac{\gamma}{2} \left[ \rho - 2L\left(\frac{2}{\gamma} - 1\right) \right] \lambda \left( 1 - \frac{\lambda}{2-\gamma} \right).
\end{align*}
For $\lambda \in [0,1]$ and $\gamma \in ]0,1]$, the term $\lambda \left( 1 - \frac{\lambda}{2-\gamma} \right)$ is non-negative because $1 - \frac{\lambda}{2-\gamma} \geq 1 - \frac{1}{2-\gamma} = \frac{1-\gamma}{2-\gamma} \geq 0$. 
Therefore, to satisfy the definition of $(\rho_L, \gamma)$-weak quasar-convexity, we require
\begin{equation*}
    B(\lambda) \leq \frac{\gamma \rho_L}{2} \lambda \left( 1 - \frac{\lambda}{2-\gamma} \right).
\end{equation*}
This holds if and only if $\rho_L \geq \rho - 2L\left(\frac{2}{\gamma} - 1\right)$. Since the weak quasar-convexity modulus must be non-negative by definition, we take $\rho_L = \max\left\{0,   \rho - 2L\left(\frac{2}{\gamma} - 1\right)\right\}$, which completes the proof.
\end{proof}

\begin{remark}
 Lemma \ref{lem:quadratic-perturbation} shows why a naive quadratic perturbation in the function of Example \ref{example:02} cannot by itself yield an instance with $\mu>\rho$ while preserving failure of quasiconvexity: as $L$ grows, $\rho_L$ is driven toward $0$ exactly in the regime needed for $\mu_L>\rho_L$, at which point the perturbed function becomes genuinely (star-)convex. This motivated the alternative, piecewise construction of Example \ref{exam:qg-radial}.
\qed
\end{remark}

Finally, we present an example in which Theorem \ref{thm:ppm-linear} applies, and no other result in the literature does. To that end, we first show the following.

%________________________%
%       Example          %
%________________________%
\begin{example}[A Radial Function with $\mu>\rho$]\label{exam:qg-radial} 
Let $\varphi:\mathbb{R}\to\mathbb{R}$ be the even function defined, for $t\geq0$, by
\begin{equation}\label{eq:phi-def}
\varphi(t) := \begin{cases}
 t^2, & 0\leq t\leq 1, \\[2mm]
 1+2(t-1)-\dfrac18(t-1)^2, & 1\leq t\leq 13, \\[2mm]
 7-(t-13)+(t-13)^2, & t\geq 13,
\end{cases}
\qquad \varphi(-t):=\varphi(t).
\end{equation}
Clearly, $\varphi$ is continuously differentiable on $\mathbb{R}$. Indeed, at $t=1$, both pieces give value $1$ and slope $2$, while at $t=13$, both pieces give value $7$ and slope $-1$, and $\varphi$ is smooth and even at $t=0$. We now record the main properties of $\varphi$.
\begin{itemize}
\item[$(i)$] {\rm Unique global minimizer.} $\varphi (0) = 0$, and $\varphi(t)>0$ for every $t\neq0$: on $[0, 1]$, $\varphi(t)=t^2>0$ for $t>0$; on $[1, 13]$, $\varphi$ increases from $\varphi(1) = 1$ to a strict local maximum $\varphi (9) = 9$ (since $\varphi' (t) = 2-\tfrac14 (t-1)$ vanishes exactly at $t=9$), then decreases to $\varphi(13) = 7$; on $[13, +\infty[$, $\varphi$ decreases further to a local minimum $\varphi(13.5) = \tfrac{27}{4}$ (since $\varphi'(t) = -1+ 2(t-13)$ vanishes at $t=13.5$), then increases to $+\infty$. In particular $\varphi(t) \geq \tfrac{27}{4}>0$ for all $t\geq13$. Hence, $\{0\} = \argmin_{\mathbb{R}} \varphi$.

\item[$(ii)$] {\rm Global weak convexity.} On $]0, 1[$, $\varphi''(t) = 2$; on $]1, 13[$, $\varphi'' (t) = -\tfrac14$; on $]13, + \infty[$, $\varphi'' (t) = 2$, with the corresponding formulas on the negative half-line by evenness. Since $\varphi + \tfrac18 \lvert \cdot \rvert^2$ is continuously differentiable and piecewise quadratic with nonnegative second derivative on each quadratic piece, its derivative is nondecreasing. Hence it is convex. Thus $\varphi$ is $\tfrac14$-weakly convex on $\mathbb{R}$. In particular, it is $(\rho,1)$-weakly quasar-convex with respect to $0$ for every $\rho\geq\tfrac14$, by \eqref{eq:quasar-wconvex}.

\item[$(iii)$] {\rm Local quadratic growth.} Since $\varphi(t) = t^2$ exactly on $[-1, 1]$, and $\varphi(t) \geq 1$ for every $t$ with $\lvert t \rvert \geq 1$ (as $\varphi$ increases on $[1,9]$ from $1$, decreases on $[9,13]$ to $7$, and is bounded below by $\tfrac{27}{4} > 1$ on $[13, + \infty[$), the sublevel set $\{t \in \mathbb{R}:   \varphi(t) \leq \zeta\}$ is contained in $[-1, 1]$ for every $\zeta \in ]0, 1[$, and on that set $\varphi (t) = t^2$ exactly. Hence $\varphi$ satisfies the $\mu$-local quadratic growth condition \eqref{eq:QG} at $\overline x=0$ with $\mu=2$, for every $\zeta \in ]0, 1[$.

\item[$(iv)$] {\rm Failure of (star) quasi-convexity.} The point $t=9$ lies on the segment joining $0$ and $13.5$, and $\varphi(9) = 9 >   \frac{27}{4} = \max\{\varphi(0),\varphi(13.5)\}$, which directly violates quasi-convexity and, in particular, star quasi-convexity with respect to $0$.

\item[$(v)$] {\rm Failure of quasar-convexity, for every $\gamma\in ]0,1]$.} Taking $x=13.5$ and $\lambda=\tfrac13$ (so that $(1-\lambda)x=9$), $\gamma$-quasar-convexity with respect to $0$ would require $\varphi(9)\leq\left(1-\tfrac\gamma3 \right) \varphi(13.5)$. Since $\varphi(13.5)=\tfrac{27}{4}$ and $1-\tfrac\gamma3\leq1$ for every $\gamma\in ]0,1]$, we get $\left(1-\frac\gamma3\right) \varphi (13.5)  \leq  \frac{27}{4}  =  6.75  <  9  =  \varphi(9)$, a contradiction for every $\gamma\in ]0,1]$. Hence $\varphi$ is not $\gamma$-quasar-convex with respect to $0$, for any $\gamma\in ]0,1]$.
\qed
\end{itemize}
%\qed
\end{example}

Now we present the desired example.

%________________________%
%       Example          %
%________________________%
\begin{example}[An Application for PPM] \label{exam:qg-application}
Let $f(s)=2-\sqrt{|s|}$ be as in Example~\ref{exam:large}, and let $\varphi$ be the function of Example~\ref{exam:qg-radial}. Divide $\mathbb{R}^2$ into the four closed regions delimited by the lines $x_2=x_1$ and $x_2=-x_1$, namely
\begin{align*}
\mathcal{A} &:= \left\{(x_1,x_2)\in\mathbb{R}^2:~ x_1 \geq 0,  -x_1 \leq x_2 \leq x_1\right\}, &
\mathcal{B} &:= \left\{(x_1,x_2)\in\mathbb{R}^2:~ x_2 \geq 0,  -x_2 \leq x_1 \leq x_2\right\}, \\
\mathcal{C} &:= \left\{(x_1,x_2)\in\mathbb{R}^2:~ x_1 \leq 0,  x_1 \leq x_2 \leq -x_1\right\}, &
\mathcal{D} &:= \left\{(x_1,x_2)\in\mathbb{R}^2:~ x_2 \leq 0,  x_2 \leq x_1 \leq -x_2\right\}.
\end{align*}
Define $H:\mathbb{R}^2\to\mathbb{R}$ by $H(0,0):=0$, and for $(x_1,x_2)\neq(0,0)$,
\begin{equation*}
H(x_1,x_2) :=
\begin{cases}
\varphi(x_1)~ f\!\left(\dfrac{x_2}{x_1}\right), & (x_1,x_2) \in \mathcal{A}, \\[3mm]
\varphi(x_2)~ f\!\left(\dfrac{x_1}{x_2}\right), & (x_1,x_2) \in \mathcal{B}, \\[3mm]
\varphi(-x_1)~ f\!\left(\dfrac{x_2}{-x_1}\right), & (x_1,x_2) \in \mathcal{C}, \\[3mm]
\varphi(-x_2)~ f\!\left(\dfrac{x_1}{-x_2}\right), & (x_1,x_2) \in \mathcal{D}.
\end{cases}
\end{equation*}
The four formulas agree on the shared boundary rays $x_2=\pm x_1$, so $H$ is well defined and continuous on $\mathbb{R}^2$.

\medskip
\noindent
\emph{Unique global minimizer.} Since $\varphi(s)=0$ if and only if $s=0$ by Example~\ref{exam:qg-radial}$(i)$ and $f>0$ on $[-1,1]$, the same argument as in Example~\ref{exam:large} shows $H(x)=0$ if and only if $x=(0,0)$, thus $\{(0,0)\}=\argmin_{\mathbb{R}^2}H$.

\medskip
\noindent
\emph{Local quadratic growth with $\mu_H=1$.} On $\mathcal{A}$, using $\varphi(t)=t^2$ for $t\in[0,1]$ and $f\geq1$, we have
\begin{equation*}
 H(x_1,x_2) = \varphi(x_1) f\!\left( \frac{x_2}{x_1} \right) \geq \varphi(x_1) \geq x_1^2 \geq \frac12 \|x\|^2, \qquad \forall  (x_1, x_2) \in \mathcal{A},\ 0< x_1 \leq 1,
\end{equation*}
using $|x_2|\leq x_1$ on $\mathcal{A}$. The same bound holds trivially at $(0,0)$. On the other hand, for every $(x_1,x_2)\in \mathcal{A}$ with $x_1 \ge 1$ we have $H(x_1,x_2) \ge \varphi(x_1) \ge 1$, by Example \ref{exam:qg-radial}(iii) and $f\ge1$. Hence, for every $\zeta \in \, ]0, 1[$, the sublevel set $\{H \le \zeta\} \cap \mathcal{A}$ is contained in $\{x_1 < 1\}$, where the bound above holds.

\medskip
\noindent
\emph{Weak quasar-convexity with $\rho_H = \tfrac12$.} On any ray through the origin, $x_2/x_1 = c \in [-1,1]$ constant, so restricted to that ray, $H(t,ct) = f(c) \varphi(t)$ is a positive scalar multiple of $\varphi$, with $f(c) \in [1,2]$. Since $\varphi$ is $\tfrac14$-weakly convex on $\mathbb{R}$ by Example \ref{exam:qg-radial}$(ii)$, multiplying by $f(c)\leq2$ gives that $t \mapsto f(c) \varphi(t)$ is $\tfrac12$-weakly convex, hence $\left(\tfrac12, 1\right)$-weakly quasar-convex with respect to $t=0$, for every $c \in [-1, 1]$. As in Example~\ref{exam:large}$(i)$, this shows $H$ is $\left(\tfrac12, 1\right)$-weakly quasar-convex with respect to $(0, 0)$ on $\mathbb{R}^2$.

\medskip
\noindent
\emph{Not weakly convex for any modulus.} Fix $\rho\geq0$ and, for $s\in ]0,1]$, consider $(1,-s),(1,s)\in\mathbb{R}^2$ and their midpoint $(1,0)$. Since $f$ is even, $H(1,\pm s)=\varphi(1)f(s)$ and $H(1,0)=\varphi(1)f(0)$. As in Example~\ref{exam:large}$(ii)$, weak convexity with modulus $\rho$ would require $\varphi(1)\sqrt{s}\leq\tfrac\rho2 s^2$, using $\varphi(1)=1>0$ and $f(0)-f(s)=\sqrt s$; this fails as $s\to0^+$ for any fixed $\rho\geq0$. Hence, $H$ is not weakly convex on $\mathbb{R}^2$ for any modulus.

\medskip
\noindent
\emph{Neither (star) quasi-convexity nor $\gamma$-quasar-convexity for any $\gamma\in ]0,1]$}. It suffices to consider $H$ restricted to $(t,0)$, $t\geq0$, where $H(t,0)=f(0)\varphi(t)=2\varphi(t)$. Using Example~\ref{exam:qg-radial}$(iv)$--$(v)$, we get
\begin{equation*}
 H(0,0)=0, \qquad H(9,0)=18, \qquad H(13.5,0) = \frac{27}{2}.
\end{equation*}
Since $H (9,0) = 18 > \tfrac{27}{2} = \max\{H(0,0), H(13.5,0)\}$, and $(9,0)$ lies on the segment joining $(0,0)$ and $(13.5,0)$, this violates quasi-convexity and star quasi-convexity with respect to $(0,0)$ (see \cite[Theorem 12]{PQK-L}). Furthermore, taking $x=(13.5,0)$, $\lambda = \tfrac13$, the same computation as in Example \ref{exam:qg-radial}$(v)$ (scaled by the common factor $2$) shows $H(9,0)=18 > \left(1-\tfrac\gamma3 \right)\cdot\tfrac{27}{2}$ for every $\gamma \in   ]0,1]$, so $H$ is not $\gamma$-quasar-convex with respect to $(0,0)$, for any $\gamma \in   ]0,1]$.

\medskip

Therefore, $H$ is $\left(\tfrac12,1\right)$-weakly quasar-convex with respect to its unique global minimizer $\overline{x}=(0,0)$, and satisfies the $1$-local quadratic growth condition \eqref{eq:QG} at $\overline{x}$ for every $\zeta \in ]0, 1[$, with
\begin{equation*}
 \mu_H = 1  >  \tfrac12 = \rho_H,
\end{equation*}
yet $H$ is neither weakly convex, nor quasi-convex, nor star quasi-convex, nor quasar-convex, for any modulus. As a consequence, provided the initial point $x^0$ of the PPM satisfies $H(x^0)<1$ (so that the trajectory remains within the region of validity $\zeta \in ]0, 1[$), Theorem~\ref{thm:ppm-linear} applies to $H$, while no other result in the literature does.
\qed
\end{example}

%%==========================%%
%%       Conclusion         %%
%%==========================%%
\section{Conclusion}\label{sec:Conclusion}

We introduced the class of weakly quasar-convex functions, unifying weak convexity and quasar-convexity under a single $(\rho,\gamma)$-parametrized inequality anchored at a global minimizer. We developed its basic properties and showed that the class strictly contains the union of its two parent notions. We further proved that the Moreau envelope preserves weak quasar-convexity with an explicit modulus update, yielding coercivity and characterization results for the associated proximity operator. Building on this machinery, we established $Q$-linear convergence of the proximal point algorithm under a local quadratic growth condition, together with matching iteration complexity and a computable, minimizer-free stopping rule based on the proximal residual.

\section{Declarations}

\subsection{Availability of supporting data}

 No data sets were generated during the current study. 

\subsection{Author Contributions}

 All authors contributed equally to the study conception, design and implementation and wrote and corrected the manuscript.

\subsection{Conflicts of Interest}

 There are no conflicts of interest or competing interests related to this manuscript.

\subsection{Funding}

 A. Khademi acknowledges financial support from IVADO and the Canada First Research Excellence Fund, and F. Lara from ANID-Chile through project Fondecyt Regular 1241040.

%---------------------%
\bibliographystyle{abbrv}  % or use: plain, abbrv, alpha, unsrt, etc.
\bibliography{bibliography}
%---------------------%

\end{document}